\documentclass[10pt,twoside]{amsart}
\usepackage{amsmath,amssymb,amsthm}
\usepackage{mathrsfs}		
\usepackage{graphicx}
\usepackage[curve]{xypic}
\usepackage{leftidx}		

\usepackage[colorlinks=true, urlcolor=blue,bookmarks=true,bookmarksopen=true, citecolor=blue]{hyperref}

\numberwithin{equation}{section}

\newtheorem{theorem}{Theorem}[section]
\newtheorem{assumption}[theorem]{Assumption}

\newtheorem{corollary}[theorem]{Corollary}

\newtheorem{lemma}[theorem]{Lemma}
\newtheorem{prop}[theorem]{Proposition}
\newtheorem*{theorem*}{Theorem}
\newtheorem*{prop*}{Proposition}

\theoremstyle{remark}
\newtheorem{defn}[theorem]{Definition}
\newtheorem{example}[theorem]{Example}
\newtheorem{remark}[theorem]{Remark}

\newtheorem*{defn*}{Definition}
\newtheorem*{remark*}{Remark}

\def \begineq{\begin{equation}}
\def \endeq{\end{equation}}

\def \bb{\mathbb}

\def \mc{\mathcal}
\def \mf{\mathfrak}

\newcommand{\Cnabla}[1]{{\leftidx{^C}{\nabla}{^{#1}}}}

\newcommand{\half}[1]{\frac{#1}{2}}

\newcommand{\norm}[1]{{\left|\!\left|#1\right|\!\right|}}

\renewcommand{\tilde}{\widetilde}

\def \CC{{\bb{C}}}

\def \GG{{\bb G}}
\def \JJ{{\bb{J}}}

\def \PP{{\bb{P}}}

\def \RR{{\bb{R}}}

\def \TT{{\bb{T}}}

\def \ZZ{{\bb{Z}}}

\def \FFC{{\mc F}}

\def \LLC{{\mc L}}

\def \({\left(}
\def \){\right)}
\def \<{\left\langle}
\def \>{\right\rangle}

\def \bar{\overline}
\def \deg{\mathrm{deg}}
\def \dsum{\oplus}

\def \half{\frac{1}{2}}
\def \inter{\cap}
\def \into{\hookrightarrow}
\def \Laplacian{{\Delta}}

\def \tensor{\otimes}
\def \union{\cup}
\def \vargeq{\geqslant}
\def \varleq{\leqslant}
\def \weakto{\rightharpoonup}

\def \Ad{{\rm Ad}}

\def \End{{\rm End}}

\def \Herm{{\rm Herm}}

\def \id{{\rm Id}}
\def \img{{\rm img}}
\def \ind{{\rm ind}}

\def \rk{{\rm rk }}
\def \Span{{\rm Span}}

\def \tr{{\rm tr}}
\def \vol{{\rm vol}}
\def \Vol{{\rm Vol}}

\renewcommand{\1}{1\!\!1}

\def \qed{\hfill $\square$ \vspace{0.1in}}

\begin{document}

\title{A Kobayashi-Hitchin correspondence for $I_\pm$-holomorphic bundles}

\author{Shengda Hu}
\address{{\it Shengda Hu}: Department of Mathematics, Wilfrid Laurier University, 75 University Ave. West, Waterloo, Canada}
\email{shu@wlu.ca}
\author{Ruxandra Moraru}
\address{{\it Ruxandra Moraru}: Department of Pure Mathematics, University of Waterloo, 200 University Ave. West, Waterloo, Canada}
\email{moraru@math.uwaterloo.ca}
\author{Reza Seyyedali}
\address{{\it Reza Seyyedali}: Department of Pure Mathematics, University of Waterloo, 200 University Ave. West, Waterloo, Canada}
\email{rseyyeda@uwaterloo.ca}


\thanks{Shengda Hu was partly supported by an NSERC Discovery Grant. Ruxandra Moraru was partially supported by an NSERC Discovery Grant.}

\abstract
In this paper, we introduce the notions of $\alpha$-Hermitian-Einstein metric and $\alpha$-stability for $I_\pm$-holomorphic vector bundles on bi-Hermitian manifolds. 
Moreover, we establish a Kobayashi-Hitchin correspondence for $I_\pm$-holomorphic vector bundles on bi-Hermitian manifolds. 
Examples of such vector bundles include generalized holomorphic bundles on generalized K\"ahler manifolds. 
We also show that $\alpha$-stability of a vector bundle, in this sense, can depend on the parameter $\alpha$.
\endabstract

\maketitle

\section{Introduction}\label{sect:intro}

A {\em bi-Hermitian structure} on a manifold $M$ consists of a triple $(g,I_+,I_-)$ where $g$ is a Riemannian metric on $M$ and $I_\pm$ are integrable complex structures on $M$ that are both orthogonal with respect to $g$. Let $(M,g,I_+,I_-)$ be a bi-Hermitian manifold. In this paper, we study the stability properties of complex vector bundles on $M$  endowed with two holomorphic structures $\bar{\partial}_+$ and $\bar{\partial}_-$ with respect to the complex structure $I_+$ and $I_-$, respectively. Our motivation for studying such bundles comes from generalized complex geometry.

\subsection{Motivation}\label{subsect:motivation}
In generalized complex geometry, generalized holomorphic bundles are analogues of holomorphic vector bundles on complex manifolds, introduced by Gualtieri in \cite{Gualtieri0703} 
(see \S \ref{subsect:genholobundles} for a definition). For instance, on a complex manifold $M$, generalized holomorphic bundles correspond to co-Higgs bundles, which are pairs 
$(E, \varphi)$ consisting of a holomorphic vector bundle $E$ on $M$ and a holomorphic bundle map $\varphi: E \rightarrow E \otimes T_X$. Some of the general properties of co-Higgs bundles were studied by Hitchin in \cite{Hitchin11}. Moreover, moduli spaces of stable co-Higgs bundles were studied by Rayan on $\mathbb{P}^1$ \cite{Rayan10,Rayan11} and on  $\mathbb{P}^2$ \cite{Rayan13}, by Vicente-Colmenares on Hirzebruch surfaces \cite{Vicente-Colmenares}, and by Gualtieri-Hu-Moraru 
on Hopf surfaces \cite{Gualtieri-Hu-Moraru}. 

An interesting problem to consider is finding a good notion of stability for generalized holomorphic bundles on any generalized complex manifold. 
Given their relationship with bi-Hermitian geometry, a natural place to start is with generalized K\"ahler manifolds. 
Gualtieri has indeed shown \cite{Gualtieri04, Gualtieri10} that any generalized K\"ahler structure $(\JJ,\JJ')$ on a manifold $M$ is equivalent to a bi-Hermitian structure $(g,I_+,I_-)$ such that
\begin{equation}\label{eq:torsion}
d^c_+ \omega_+ = -d^c_- \omega_- = \gamma
\end{equation}
for some closed $3$-form $\gamma \in \Omega^3(M)$, where $\omega_\pm(\cdot,\cdot) = g(I_\pm \cdot,\cdot)$ are the fundamental 2-forms of $g$ and $d^c_\pm = I_\pm \circ d \circ I_\pm$ are the twisted differentials with respect to $I_\pm$. Let $T^{0,1}_\pm M$ be the anti-holomorphic cotangent bundles of $M$ with respect to $I_\pm$.  In this case, a {\em generalized holomorphic bundle} on $(M,\JJ)$ 
(or {\em $\JJ$-holomorphic bundle} on $M$) corresponds to a triple $(V,\bar\partial_+,\bar\partial_-)$ where $V$ is a complex vector bundle on $M$ together with holomorphic structures
\[\bar\partial_\pm : C^\infty(V) \to C^\infty(V \tensor T^{0,1}_\pm M),\] 
with respect to $I_\pm$, that satisfy a commutation relation (for details, see equation \eqref{eq:commutationrelation}).
Our goal is therefore to formulate a notion of stability for $\JJ$-holomorphic bundles in terms of the bi-Hermitian structure $(g,I_+.I_-)$, in a way that preserves properties of stable bundles on complex manifolds. This would, in particular,  give us stability with respect to any generalized complex structure $\JJ$ on $M$ that fits into a generalized K\"ahler pair $(\JJ,\JJ')$.

\subsection{Objects of study}\label{subsect:objects} We begin by considering a more general situation. Let $(M, g, I_+, I_-)$ be a compact bi-Hermitian manifold. Note that we are not assuming that $g$ satisfies equation \eqref{eq:torsion}  so that the bi-Hermitian structure $(g,I_+,I_-)$ may not come from a generalized K\"ahler structure. 
\begin{defn*}
 An \emph{$I_\pm$-holomorphic bundle} on $M$  consists of a triple $(V,\bar\partial_+,\bar\partial_-)$ where $V$ is a complex vector bundle on $M$ and $\bar\partial_\pm$ 
 are $I_\pm$-holomorphic structures on $V$. Moreover, we call the pair $(\bar\partial_+,\bar\partial_-)$ an \emph{$I_\pm$-holomorphic structure} on $V$.
\end{defn*}

\noindent
Consequently, on a generalized K\"ahler manifold $(M,\JJ,\JJ')$ with associated bi-Hermitian structure $(g,I_+,I_-)$, a $\JJ$-holomorphic bundle corresponds to an $I_\pm$-holomorphic bundle that satisfies a commutation relation (for details, see \S \ref{sect:genholobundles}).

The main result of this paper is a Kobayashi-Hitchin correspondence for $I_\pm$-holomorphic bundles. On a compact complex manifold endowed with a Gauduchon metric, the {\em Kobayashi-Hitchin correspondence} states that a holomorphic vector bundle admits a Hermitian-Einstein metric if and only if it is polystable.
This correspondence was first proven for Riemann surfaces by Narasimhan-Seshadri \cite{NarasimhanSeshadri}, then for K\"ahler manifolds by Donaldson-Uhlenbeck-Yau \cite{Donaldson1983, Donaldson1985, UhlenbeckYau86, Donaldson1987} and for complex manifolds with Gauduchon metrics by Buchdahl-Li-Yau \cite{Buchdahl, Li-Yau}.  We refer the reader to \cite{LubkeTeleman1995} for a general reference on the Kobayashi-Hitchin correspondence. Before stating the correspondence for bi-Hermitian manifolds, let us first describe what we mean by Hermitian-Einstein and polystable in this context.

\smallskip
In the remainder, we make the following assumption on the metric $g$.
\begin{assumption}\label{assump:Gauduchon}
The metric $g$ is Gauduchon with respect to both $I_+$ and $I_-$, that is, $dd^c_\pm \omega_\pm = 0$, and $\vol_g = \frac{1}{n!} \omega_\pm^n$.
 \end{assumption}
 \noindent
We note that this assumption is not too restrictive. It is automatically satisfied on {\em even} generalized K\"ahler 4-manifolds. Moreover, such metrics exist in higher dimension, for example on real compact Lie groups \cite{HuPre}. The assumption is, however, essential to our discussion of polystability.

\smallskip
Let $(V,\bar\partial_+, \bar\partial_-)$ be an $I_\pm$-holomorphic bundle on $M$ and $h$ be a Hermitian metric on $V$. Let $F_\pm$ be the curvatures of the Chern connections $\nabla^C_{\pm}$ on $V$ corresponding to the holomorphic structures $\bar\partial_\pm$. Motivated by Hitchin \cite{Hitchin11}, we introduce the following \emph{$\alpha$-Hermitian-Einstein equation}, where $\alpha \in (0,1)$ and $\lambda \in\RR$:
\begin{equation}\label{eq:alphaHEintro}
 \sqrt{-1}(\alpha F_+\wedge \omega_+^{n-1} + (1-\alpha) F_-\wedge \omega_-^{n-1}) = (n-1)!\lambda \id_V \vol_g.
\end{equation}
\noindent
Note that this equation is a natural generalisation of the Hermitian-Einstein equation on a complex manifold, endowed with a Gauduchon metric, to the bi-Hermitian setting; as in the complex case, we have:
\begin{defn*}
The Hermitian metric $h$ on $V$ is called \emph{$\alpha$-Hermitian-Einstein} if the corresponding pair $(\nabla^C_+, \nabla^C_-)$ of Chern connections satisfies the $\alpha$-Hermitian-Einstein equation \eqref{eq:alphaHEintro}.
\end{defn*}

Now, for polystability, we need notions of degree and of coherent subsheaf of an $I_\pm$-holomorphic bundle $(V,\bar\partial_+,\bar\partial_-)$. First note that, since we are assuming $g$ to be 
Gauduchon with respect to both $I_+$ and $I_-$, we can associate to $V$ two degrees $\deg_\pm(V)$ and two slopes $\mu_\pm(V)$ in the standard way:
 \[\deg_\pm(V) =\frac{\sqrt{-1}}{2\pi} \int_M \tr(F_\pm) \wedge \omega_\pm^{n-1}\] 
 and 
 \[\mu_\pm(V) = \frac{\deg_\pm(V)}{\rk V}.\] 
We then define the {\em $\alpha$-degree $\deg_\alpha(V)$} and the {\em $\alpha$-slope $\mu_\alpha(V)$} of $(V,\bar\partial_+,\bar\partial_-)$, for any $\alpha \in (0,1)$, as
\[ \deg_\alpha(V) := \alpha \deg_+(V)+ (1-\alpha) \deg_- (V)\]
and
\[\mu_\alpha(V) := \alpha \mu_+(V) + (1-\alpha)\mu_-(V),\]
respectively. As for coherent subsheaves of $(V,\bar\partial_+,\bar\partial_-)$, they are defined in \S \ref{sect:biholbundles} (see definition \ref{defn:coherentsheaves}).

We are now in a position to define the analogue of classical slope stability (Mumford \cite{Mumford63}) for $I_\pm$-holomorphic and $\JJ$-holomorphic bundles.
\begin{defn*}
Let $\alpha \in (0,1)$.
The $I_\pm$-holomorphic structure  $(\bar\partial_+, \bar\partial_-)$ on $V$ is called \emph{$\alpha$-stable} if, for any proper coherent subsheaf $\FFC$ of 
$(V,\bar\partial_+,\bar\partial_-)$, we have $\mu_\alpha(\FFC) < \mu_\alpha(V)$.  Furthermore,  $(\bar\partial_+, \bar\partial_-)$ is said to be \emph{$\alpha$-polystable} if it is a direct sum of $\alpha$-stable 
bundles with the same $\alpha$-slope. 
In addition, a $\JJ$-holomorphic bundle $(V,\bar\partial_+, \bar\partial_-)$ is called \emph{$\alpha$-(poly)stable} if the corresponding $I_\pm$-holomorphic structure $(\bar\partial_+, \bar\partial_-)$  is.
\end{defn*}

\subsection{Main result}
  
We can now give a precise statement of the main result of the paper.
 \begin{theorem*}[Theorem \ref{thm:setKHcorrespondence}]
 Let $(M,g,I_+,I_-)$ be a compact bi-Hermitian manifold such that $g$ is Gauduchon with respect to both $I_+$ and $I_-$, and $\vol_g = \frac{1}{n!} \omega_\pm^n$. 
 Moreover, let $(V, \bar\partial_+, \bar\partial_-)$ be an $I_\pm$-holomorphic bundle on $M$. 
 Then, $(V, \bar\partial_+, \bar\partial_-)$ admits an $\alpha$-Hermitian-Einstein metric if and only if it is $\alpha$-polystable, for any $\alpha \in (0,1)$.
 \end{theorem*}

As a corollary, we obtain a Kobayashi-Hitchin correspondence for generalized holomorphic bundles on generalized K\"ahler manifolds.
\begin{prop*}[Corollary \ref{coro:generalizeKahlerKH}]
Let $(M,\JJ,\JJ')$ be a compact generalized K\"ahler manifold whose associated bi-Hermitian structure $(g,I_+,I_-)$ is such that $g$ is Gauduchon with respect to both $I_+$ and $I_-$, 
and $\vol_g = \frac{1}{n!} \omega_\pm^n$. Moreover, let $(V, \bar\partial_+, \bar\partial_-)$ be a $\JJ$-holomorphic bundle on $M$. 
Then, $(V, \bar\partial_+, \bar\partial_-)$ admits an $\alpha$-Hermitian-Hermitian metric if and only if it is $\alpha$-polystable, for any $\alpha \in (0,1)$.
\end{prop*}

\subsection{Plan of the paper}
We begin by providing some basic facts about generalized K\"ahler geometry and generalized holomorphic bundles in section \S \ref{subsect:gengeom}, 
and introduce the notion of $I_\pm$-holomorphic bundles on bi-Hermitian manifolds in section \S \ref{subsect:genholobundles}. We then describe how these bundles relate to generalized holomorphic bundles on generalized K\"ahler manifolds in section \S \ref{subsect:genholocommutativity}.

The notions of $\alpha$-Hermitian-Einstein metric and $\alpha$-stability for $I_\pm$-holomorphic and generalized holomorphic bundles are introduced in section \S \ref{sect:biholbundles},
and some of their properties are discussed in sections \ref{sect:biholbundles} and \ref{generalizedholbun}. 

In section \S \ref{sect:Hopfexample}, we consider $I_\pm$-holomorphic and generalized holomorphic bundles on Hopf surfaces. 
These 4-manifolds admit a natural bi-Hermitian structure that corresponds to an even generalized K\"ahler structure
and thus satisfies Assumption \ref{assump:Gauduchon} (see section \S \ref{subsect:bihermitian}). 
We study line bundles with respect to this bi-Hermitian structure in section \S \ref{subsect:linebundleonHopf}, in particular showing that generalized holomorphic line bundles always exist. 
We then consider rank-2 bundles in \S \ref{subsect:rank2onhopf}, giving examples of both $\alpha$-stable and $\alpha$-unstable
$I_\pm$-holomorphic and generalized holomorphic rank-2 bundles. In particular, we show that $\alpha$-stability can depend on the choice of $\alpha$ (see examples \ref{dependence-alpha:c2=0} and 
\ref{dependence-alpha:c2=1}). 

The main result of the paper (Theorem \ref{thm:setKHcorrespondence}) is proven in section \S \ref{sect:correspondenceforIpm}, where we follow closely the presentation of Chapters 2  and 3 in \cite{LubkeTeleman1995}. The deformation theory of $I_\pm$-holomorphic and generalized holomorphic bundles will appear elsewhere \cite{HuMoraruSeyyedali}.
\medskip

\noindent
{\bf Convention.} We use the following notation throughout the paper. Let $M$ be a manifold endowed with a complex structure $I$. Then, $T_{1,0}M$ and $T_{0,1}M$ denote the bundles of holomorphic and anti-holomorphic vector fields on $M$, respectively. Furthermore, $T^{1,0}M$ and $T^{0,1}M$ denote the bundles of holomorphic and anti-holomorphic forms on $M$, respectively.

\medskip

{\bf Acknowledgements.} The authors would like to thank Marco Gualtieri and Nigel Hitchin for their interest in this work and for helpful discussions. Moreover, Shengda Hu would like to thank Bohui Chen for helpful discussions and the Yangtze Center of Mathematics at Sichuan University for their hospitality where part of research was done. 


\section{Generalized holomorphic bundles}\label{sect:genholobundles}
In this section, we give some basic facts on generalized K\"ahler geometry and generalized holomorphic bundles. More details can be found in several of Gualtieri's papers \cite{Gualtieri04, Gualtieri0703, Gualtieri0710, Gualtieri10}. We also discuss the relationship between generalized holomorphic bundles on generalized K\"ahler manifolds and 
$I_\pm$-holomorphic bundles on bi-Hermitian manifolds in sections \S \ref{subsect:genholobundles} and \S \ref{subsect:genholocommutativity}.

\subsection{Generalized geometry}\label{subsect:gengeom}
Let $M$ be a smooth $2n$-manifold, $\gamma \in \Omega^3(M)$ be a closed $3$-form, and $g$ be a Riemannian metric on $M$. 
Consider the generalized tangent bundle $\TT M = TM \dsum T^*M$, which admits a Courant-Dorfman bracket defined by $\gamma$:
$$(X+\xi)*(Y+\eta) = [X,Y] + \LLC_X\eta - \iota_Y d\xi + \iota_X\iota_Y \gamma.$$
The natural projection $\TT M \to TM$ is denoted $a$.
Recall that a \emph{generalized almost complex structure} on $M$ is an endomorphism $\JJ : \TT M \to \TT M$ that satisfies $\JJ^2 = -\1$ and is orthogonal with respect to the natural pairing
$$\<X+\xi, Y+\eta\> = \frac{1}{2}(\iota_X\eta + \iota_Y \xi).$$
Extend $*$, $\JJ$ and $\<,\>$ complex-linearly to $\TT_\CC M = \TT M\tensor_\RR \CC$.
The $i$-eigenbundle $L$ of $\JJ$ in $\TT_\CC M$ is then maximally isotropic with respect to the pairing $\<,\>$. 
Moreover, the structure $\JJ$ is \emph{integrable} and called a \emph{generalized complex structure} if, in addition, $L$ is involutive with respect to the Courant-Dorfman bracket.
Here are two basic examples of generalized complex structures.

\begin{example}\label{ex:complex}
Let $(M,I)$ be a complex manifold. The complex structure $I$ induces a natural generalized complex structure on $M$ given by
$$\JJ_I =
\begin{pmatrix}
 I & 0 \\ 0 & -I^\ast
\end{pmatrix}.
$$
\end{example}

\begin{example}\label{ex:symplectic}
Let $(M,\omega)$ be a symplectic manifold. The symplectic structure $\omega$ induces the following natural generalized complex structure 
$$\JJ_\omega =
\begin{pmatrix}
 0 & -\omega^{-1} \\ \omega & 0
\end{pmatrix}
$$
on $M$.
\end{example}

\noindent
There are nonetheless many examples of generalized complex structures that do not arise from a complex structure or a symplectic structure. 
See for example \cite{Gualtieri0710, Gualtieri0703}.

\medskip

A pair of generalized complex structures $(\JJ, \JJ')$ on $M$ defines a \emph{generalized K\"ahler structure} if they commute (that is, $\JJ\JJ' = \JJ'\JJ$) and are such that
$\GG : = -\JJ\JJ'$ is a positive definite metric on $\TT M$, called the {\em generalized K\"ahler metric}.
(In other words, the symmetric pairing
\[ G(X+\xi,Y+\eta) := \langle \GG(X+\xi),(Y+\eta) \rangle \] 
is positive definite on $\TT M$.)

\begin{example}\label{ex:Kahler}
Let $(M,I,g)$ be a K\"ahler manifold with fundamental form $\omega$, which is symplectic. We then have two natural generalized complex structures on $M$, namely, the generalized
complex structures $\JJ_I$ and $\JJ_\omega$ defined in examples \ref{ex:complex} and \ref{ex:symplectic}, respectively. In fact, the pair $(\JJ_I,\JJ_\omega)$ defines a generalized K\"ahler structure
on $M$ with generalized K\"ahler metric
$$G = 
\begin{pmatrix}
 0 & g^{-1} \\ g & 0
\end{pmatrix}.
$$
Such a generalized K\"ahler structure is called {\em trivial}.
\end{example}

\begin{remark*}
K\"ahler manifolds can admit non-trivial generalized K\"ahler structures 
\cite{Gualtieri04,Gualtieri10}. On the other hand, manifolds that do not admit K\"ahler metrics can nonetheless  sometimes admit generalized K\"ahler structures, as illustrated by the Hopf surface (see section \S \ref{sect:Hopfexample}).
\end{remark*}

We now explain how generalized K\"ahler geometry is related to bi-Hermitian geometry. 
Let $(\JJ, \JJ')$ be a generalized K\"ahler structure on $M$. It induces a bi-Hermitian structure $(g,I_+,I_-)$ on $M$ such that $d^c_\pm\omega_\pm = \pm\gamma$ as follows.
Note that $\GG^2 = \1$ so that its eigenvalues are $\pm1$.
 Let $C_{\pm}$ be the $\pm 1$-eigenbundles of $\GG$ in $\TT M$.
The projections $a_\pm :=a|_{C_\pm}: C_{\pm} \subset \mathbb{T}M \rightarrow TM$ are then isomorphisms. 
We set
\[ g(X,Y) := \pm \langle a_\pm^{-1}(X),a_\pm^{-1}(Y) \rangle\]
for all $X,Y \in C^\infty(TM)$.
Then, $g$ is a Riemannian metric on $M$ such that
$$C_\pm = \{ X \pm g(X) : X \in TM \}.$$ 
In other words, $C_\pm$ are the graphs of $\pm g$. In addition, since $\mathbb{J}$ preserves $C_\pm$, 
we can define complex structures $I_\pm$ on $M$ by restricting $\JJ$ to $C_\pm$. 
Concretely, we have
$$I_\pm(X) := a \left(\JJ (X \pm g(X))\right),$$
for $X \in C^\infty(TM)$.
Then, $I_\pm$ are clearly orthogonal with respect to $g$ and $d^c_\pm \omega_\pm = \pm \gamma$.

To see why $d^c_\pm \omega_\pm = \pm \gamma$, we first note that we have the following decompositions:
$$\TT_\CC M = C_+ \dsum C_- = L \dsum \bar L = L' \dsum \bar L' = \ell_+ \dsum \ell_- \dsum \bar \ell_+ \dsum \ell_-,$$
where $L$ and $L'$ are the $i$-eigenbundles of $\JJ$ and $\JJ'$, respectively, and 
\[ \mbox{$\ell_+ := L \inter L'$ and $\ell_- := L \inter \bar L'$.} \] 
We then have $C_\pm = \ell_\pm \dsum \bar\ell_\pm$, implying that $a$ maps  $\ell_\pm$ isomorphically onto the holomorphic tangent bundles $T^\pm_{1,0} M$ of $(M,I_\pm)$.
Thus,
\begin{equation}\label{ellpm}
\ell_\pm = \{X \pm g(X) = X \mp i\iota_X\omega_\pm : X \in T^\pm_{1,0} M\}.
\end{equation}
Furthermore, $\ell_\pm$ are involutive with respect to the Courant-Dorfman bracket because $L$ and $L'$ are.
The condition that $\ell_\pm$ are involutive is equivalent to
$$(X \mp i\iota_X \omega_\pm) * (Y \mp i\iota_Y \omega_\pm) = [X, Y] \mp i(\LLC_{X}\iota_Y\omega_\pm - \iota_{Y}d\iota_X \omega_\pm) + \iota_X\iota_Y \gamma = [X, Y] \mp i\iota_{[X,Y]} \omega_\pm,$$
which implies that
$$\pm i \iota_X\iota_Y d\omega_\pm + \iota_X\iota_Y \gamma = 0 \iff \mp i \left(d\omega_\pm\right)^{2,1}_\pm = \mp i \partial_\pm \omega_\pm= \gamma^{(2,1) + (3,0)}_\pm.$$
Since $\omega_\pm$ and $\gamma$ are both real forms, we have
$$\gamma = \mp i (\partial_\pm \omega_\pm - \bar\partial_\pm \omega_\pm) = \pm d^c_\pm \omega_\pm.$$

Conversely, given a bi-Hermitian manifold $(M,g,I_+,I_-)$ with $d^c_\pm\omega_\pm = \pm \gamma$ for some closed 3-form $\gamma$, we recover the generalized K\"ahler structure $(\mathbb{J},\mathbb{J}')$ as
\[ \JJ / \JJ' = \frac{1}{2}\begin{pmatrix}
					I_+ \pm I_- & -(\omega_+^{-1} \mp \omega_-^{-1})\\
					\omega_+ \mp \omega_- & -(I_+^* \pm I_- ^*)
					\end{pmatrix}.\]
We thus have a one-to-one correspondence between generalized K\"ahler structures and bi-Hermitian structures with torsion \cite{Gualtieri04, Gualtieri10}. In the rest of this paper, we move interchangeably between
a given generalized K\"ahler structure $(\JJ,\JJ')$ and its associated bi-Hermitian structure $(M,g,I_+,I_-)$.

\begin{example}\label{ex:Kahler-biherm}
Let $(M,I,g)$ be a K\"ahler manifold. The bi-Hermitian structure corresponding to the trivial generalized K\"ahler structure $(\mathbb{J}_I,\mathbb{J}_\omega)$ constructed in example
\ref{ex:Kahler} is then $(g,I_+=I_- = I)$.
\end{example}

\begin{remark*}
The generalized K\"ahler structure on Hopf surfaces given in section \S \ref{sect:Hopfexample} is described in terms of its associated bi-Hermitian structure.
\end{remark*}

We end this section with two definitions that will come up in the sequel.

\begin{defn}\label{defn:evenGK}
A generalized K\"ahler structure $(\JJ,\JJ')$ on $M$ is said to be {\em even} if both $a(L)$ and $a(L')$ have even codimension in $TM \otimes \CC$, where $L$ and $L'$ are the $i$-eigenbundles of
$\JJ$ and $\JJ'$, respectively.
\end{defn}

\noindent
Note that, if $M$ is real $4k$-dimensional and the generalized K\"ahler structure $(\JJ,\JJ')$ is even, then its associated bi-Hermitian structure $(g,I_+,I_-)$ is such that $\vol_g = \frac{1}{n!} \omega_\pm^n$
(see Remark 6.14 in \cite{Gualtieri04}). These manifolds therefore satisfy Assumption \ref{assump:Gauduchon}. 
The generalized K\"ahler structure on Hopf surfaces described in section \S \ref{subsect:bihermitian} is an example of even structure on a 4-manifold.

\smallskip
Recall that the \emph{Bismut connection} on a Hermitian manifold is the unique connection that is compatible with both the Riemannian metric 
and the complex structure and has totally skew-symmetric torsion. 
Let $(\JJ,\JJ')$ be a generalized K\"ahler structure on $M$ with associated bi-Hermitian structure $(g,I_+,I_-)$. 
We denote $\nabla^\pm$ the Bismut connections of $(M,g,I_\pm)$, so that $\nabla^\pm g = 0 \text{ and } \nabla^\pm I_\pm = 0$.
The connections $\nabla^\pm$ can then be expressed in terms of the Courant-Dorfman bracket as follows. Let $X, Y \in C^\infty(TM)$, then
\begin{equation}\label{Bismut}
\nabla^\pm_X Y := a\left(\left[(X\mp g(X))*(Y\pm g(Y))\right]^{\pm}\right),
\end{equation}
where $\bullet^\pm$ denote projections to $C_\pm$, respectively (see \cite{Gualtieri0710}, section \S 3, for details).
In the sequel, we also use $\nabla^\pm$ to denote the induced Bismut connection on $T^*M$.
Note that since the Bismut connections $\nabla^\pm$ preserve the complex structures $I_\pm$, respectively, they also preserve the anti-holomorphic cotangent bundles $T^{0,1}_\pm M$, respectively.

\subsection{Holomorphic bundles}\label{subsect:genholobundles}
Let $(M,\JJ)$ be a generalized complex manifold with $i$-eigenbundle $L$, and let $V$ be a complex vector bundle on $M$. 
\begin{defn}
A {\em generalized holomorphic structure} or \emph{$\JJ$-holomorphic structure} on $V$ is a flat 
$L$-connection $\bar {\mc D}$ on $V$. More explicitly, it is a derivation 
\[ \bar {\mc D} : C^\infty(V) \to C^\infty(V \tensor \bar L^*)\] 
such that $\bar{\mc D}^2 = 0$. The pair $(V, \bar {\mc D})$ is then called a {\em generalized holomorphic bundle} or {\em $\JJ$-holomorphic bundle} on $(M,\JJ)$.
\end{defn}

\begin{remark}
Note that the Leibniz Rule is given, in this case, by
\[ \bar{\mc D}(fs) = d_{\bar L}(f) \otimes s + f\bar{\mc D}(s),\]
where $d_L$ is the Lie algebroid differential of $(\bar L,*)$ (see \cite{Gualtieri04}, Definition 3.7). 
For example, 
$$d_{\bar L}(f)(v) := a(v)(f),$$ for all $v \in C^\infty(\bar L)$.
\end{remark}

Generalized holomorphic bundles were introduced by Gualtieri \cite{Gualtieri0703} and some of their properties were studied by Hitchin \cite{Hitchin11}.
Here are some examples.

\begin{example}
Let $(M,\omega)$ be symplectic manifold. In this case, $\JJ_\omega$-holomorphic bundles correspond to flat bundles where $\JJ_\omega$ is the generalized complex structure 
given in example \ref{ex:symplectic}.
\end{example}

\begin{example}\label{co-Higgs}
Let $(M,I)$ complex manifold and consider the generalized complex structure $\mathbb{J}_I $ defined in example \ref{ex:complex}.
Then $\bar L = T_{0,1} M \oplus T^{1,0} M$. A $\mathbb{J}_I$-holomorphic structure on $V$ is then a derivation
\[\bar{\mc D} = \bar{\partial} + \phi: C^\infty(V) \rightarrow C^\infty(V \otimes T^{0,1}M) \oplus C^\infty(V \otimes T_{1,0}M)\] such that
$\bar{\mc D}^2 = 0$, which is equivalent to $\bar{\partial}^2 = 0$, $\bar{\partial} \phi = 0$ and $\phi \wedge \phi = 0$. In other words,
\[ \bar{\partial}: C^\infty(V) \rightarrow C^\infty(V \otimes T^{0,1}M) \]
is a holomorphic structure on $V$, with respect to the complex structure $I$, and 
\[ \phi: C^\infty(V) \rightarrow C^\infty(V \otimes T_{1,0}M)\] 
is a holomorphic section of $\End V \otimes T_{1,0}M$ satisfying the integrability condition $\phi \wedge \phi = 0$.
Hence, $\mathbb{J}_I$-holomorphic bundles correspond to co-Higgs-bundles $(V,\bar{\partial},\phi)$ with $\phi$ the {\em Higgs field}.
\end{example}


For the remainder of the paper, we assume that the generalized complex structure $\JJ$ is part of a generalized K\"ahler pair $(\JJ,\JJ')$ on $M$ 
with associated bi-Hermitian structure $(g,I_+.I_-)$. Referring to section \S \ref{subsect:gengeom}, we have $\bar L = \bar \ell_+ \oplus \bar \ell_-$.
Similarly to generalized holomorphic structures, 
we can define flat $\bar \ell_\pm$-connections on $V$ as derivations 
\[ \bar {\mc D}_\pm : C^\infty(V) \to C^\infty(V \tensor \bar \ell_\pm^*)\]
such that $\bar{\mc D}_\pm^2 = 0$. In this case, any $\JJ$-holomorphic structure $ \bar {\mc D}$ can be written as a sum
\[  \bar {\mc D} =  \bar {\mc D}_+ +  \bar {\mc D}_- \]
of $\bar\ell_\pm$-connections $\bar{\mc D}_\pm$ on $V$. By composing $\bar{\mc D}_\pm$ with the isomorphisms $a : \bar\ell_\pm \to T^\pm_{0,1}M$, we obtain
 partial connections 
\[ \bar\partial_\pm : C^\infty(V) \to C^\infty(V \tensor T_\pm^{0,1} M)\] 
given by 
\[\bar\partial_{\pm,X}(v) := \bar{\mc D}_{\pm,s}(v),\]
for all $X \in C^\infty(T^\pm_{0,1}M)$ and $v \in C^\infty(V)$, where $s \in C^\infty(\bar\ell_\pm)$ is such that $X = a(s)$.
In particular, $\bar{\mc D}_\pm$ is flat if and only if $\bar\partial_\pm$ is flat, in which case, $(V,\bar\partial_\pm)$ is an $I_\pm$-holomorphic bundle.

\begin{defn}\label{Ipmbun}
An {\em $I_\pm$-holomorphic structure} on $V$ is defined to be a pair $(\bar \partial_+,\bar\partial_-)$
of holomorphic structures  on $V$ with respect to the complex structures $I_+$ and $I_-$, respectively. In this case, the triple
$(V, \bar \partial_+,\bar\partial_-)$ is called an {\em $I_\pm$-holomorphic bundle}.
\end{defn}

\begin{prop}\label{prop:genholoisbiholo}
 Any $\JJ$-holomorphic bundle induces an $I_\pm$-holomorphic bundle  on $(M, g, I_+, I_-)$.
\end{prop}
\begin{proof}
Let $\bar{\mc D} : C^\infty(V) \to C^\infty(V \tensor \bar L^*)$ be a flat $\bar L$-connection. Since $\bar L = \bar \ell_+ \dsum \bar \ell_-$, the connection decomposes as
$\bar{\mc D} = \bar{\mc D}_+ \dsum \bar{\mc D}_- $ with $\bar{\mc D}_\pm : C^\infty(V) \to C^\infty(V \tensor \bar \ell_\pm^*)$.
More explicitly, let $s_\pm \in C^\infty(\bar\ell_\pm)$ and write $s = s_+ + s_- \in C^\infty(\bar L)$. Then, for $v \in C^\infty(V)$,
$$\bar{\mc D}_s(v) := \<\bar{\mc D}(v), s\> = \bar{\mc D}_{s_+}(v) + \bar{\mc D}_{s_-}(v).$$ 
Define 
$$\bar{\mc D}_{\pm, s_\pm} (v) = \bar{\mc D}_{s_\pm}(v).$$
The flatness of $\bar{\mc D}$ is equivalent to the following identity for any $s, t \in C^\infty(\bar L)$ and $v \in C^\infty(V)$:
\begin{equation}\label{eq:flatnessdiffequation}
 \bar{\mc D}_s\bar{\mc D}_t(v) - \bar{\mc D}_t\bar{\mc D}_s(v) - \bar{\mc D}_{s*t}(v) = 0.
\end{equation}
By setting $s_\bullet = t_\bullet = 0$, $\bullet = \pm$, in \eqref{eq:flatnessdiffequation}, we see that $\bar{\mc D}_\pm$ are flat as well.
\end{proof}

\begin{example}\label{exple:trivialholomorphic}
Consider the trivial line bundle $M \times \CC$ on $M$ and its constant section $u(z) = (z,1)$. 
  The trivial $\bar L$-connection $\bar{\mc D}$ on $M \times \CC$ is given by $\bar{\mc D}(u) = 0$. 
  Clearly, $\bar{\mc D}^2 = 0$ and the corresponding partial connections $\bar\partial_\pm$ coincide with the standard complex operators $\bar\partial_\pm$ on functions:
 \[\bar\partial_\pm (fu) = \bar\partial_\pm f \tensor u.\]
\end{example}

\noindent
For examples of non-trivial $\JJ$-holomorphic and $I_\pm$-holomorphic bundles, see section \S \ref{sect:Hopfexample}.

\subsection{Commutation relation}\label{subsect:genholocommutativity}
Let $(V, \bar\partial_+, \bar\partial_-)$ be an $I_\pm$-holomorphic bundle; the $I_\pm$-holomorphic structures $\bar\partial_\pm$ induce 
$\bar\ell_\pm$-connections $\bar{\mc D}_\pm$ by
\[\bar{\mc D}_{\pm,s}(v) := \bar\partial_{\pm,a(s)}(v),\] 
for all $s \in C^\infty(\bar\ell_\pm)$ and $v \in C^\infty(V)$,
via the isomorphisms $a : \bar\ell_\pm \to T_{0,1}^\pm M$. Define 
\[ \bar{\mc D} = \bar{\mc D}_+ \dsum \bar{\mc D}_- : C^\infty(V) \to C^\infty(V\tensor \bar L^*).\]
Then $(V,\bar{\mc D})$ is a $\JJ$-holomorphic bundle if and only if \eqref{eq:flatnessdiffequation} holds.

Since the Bismut connections $\nabla^\pm : C^\infty(T^*M) \to C^\infty(T^*M \tensor T^*M)$ preserve the sub-bundles $T^{0,1}_\pm M$, we consider the component $\bar\delta_\pm$ of $\nabla^\pm$ defined by the following composition:
\begin{equation}\label{eq:partialBismut}
 \bar\delta_\pm := \nabla^+|_{T^{0,1}_\pm M} : C^\infty(T^{0,1}_\pm M) \to C^\infty(T^*M \tensor T^{0,1}_\pm M) \to C^\infty(T^{0,1}_\mp M \tensor T^{0,1}_\pm M).
\end{equation}
For example, for $\alpha_+ \in C^\infty(T^{0,1}_+ M)$, $\bar\delta_+ \alpha_+ \in C^\infty(T^{0,1}_- M \tensor T^{0,1}_+ M)$. 
We extend the partial connections $\bar\partial_\pm$ on $V$ to $V \tensor T^{0,1}_\pm M$ as follows:
\begin{equation}\label{eq:Bismutextension}
 \bar\partial_\pm : C^\infty(V \tensor T^{0,1}_\pm M) \to C^\infty(V \tensor T^{0,1}_\mp M \tensor T^{0,1}_\pm M): \bar\partial_\pm(v \tensor \alpha_\pm) = \bar\partial_\pm(v) \tensor \alpha_\pm + v \tensor \bar\delta_\pm(\alpha_\pm).
\end{equation}
\begin{prop}\label{prop:bihermitianholobundle}
 $\JJ$-holomorphic bundles are exactly the $I_\pm$-holomorphic bundles that also satisfy the commutation relation
\begin{equation}\label{eq:commutationrelation}
 [\bar\partial_+, \bar\partial_-] := \bar\partial_+\bar\partial_- + \bar\partial_-\bar\partial_+ = 0,
\end{equation}
where $T^{0,1}_+ M \tensor T^{0,1}_- M$ is identifies with $T^{0,1}_- M \tensor T^{0,1}_+ M$ via $\alpha\tensor\beta \mapsto -\beta\tensor \alpha$.
\end{prop}
\begin{proof}
For $s \in C^\infty(\bar \ell_+)$ and $t \in C^\infty(\bar \ell_-)$, the equation \eqref{eq:flatnessdiffequation} becomes
\[\bar{\mc D}_{+, s}\bar{\mc D}_{-, t}(v) - \bar{\mc D}_{-, t}\bar{\mc D}_{+, s}(v) - \bar{\mc D}_{s*t} (v) = 0.\]
Since $s, t \in C^\infty(\bar\ell_+ \dsum \bar\ell_-) = C^\infty(\bar L)$ and $\bar{L}$ is involutive and isotropic, we see that $s*t = -t*s$ and
\[ [X, Y] = a(s*t) = a\left((s*t)^- - (t*s)^+\right) = \nabla^-_X Y - \nabla^+_Y X,\]
where $X = a(s) \in C^\infty(T_{0,1}^+ M)$ and $Y = a(t) \in C^\infty(T_{0,1}^-M)$. Thus, we get \begin{equation}\label{eq:PairedCommutativity}
\bar\partial_{+,X}\bar\partial_{-,Y}(v) - \bar\partial_{-,\nabla^-_X Y}(v)- \bar\partial_{-,Y}\bar\partial_{+,X}(v) + \bar\partial_{+,\nabla^+_Y X}(v) = 0,
\end{equation}
which is equivalent to \eqref{eq:commutationrelation}.
\end{proof}

\begin{remark}
Note that not every $I_\pm$-holomorphic bundle satisfies the commutation relation \eqref{eq:commutationrelation} (see examples \ref{c2=0} and \ref{dependence-alpha:c2=1}).
\end{remark}

\begin{example}\label{exple:linebundlecommuteoperator}
Let $(V,\bar\partial_+,\bar\partial_-)$ be an $I_\pm$-holomorphic line bundle and $u$ be a local section of $V$. 
 Then, $\bar\partial_\pm(u) = u \tensor \alpha_\pm$, for some $\alpha_\pm \in C^\infty(T^{0,1}_\pm M)$, and
\[ [\bar\partial_+, \bar\partial_-](u) =  \bar\partial_+(u\tensor\alpha_-) + \bar\partial_-(u \tensor \alpha_+) = u\tensor (\bar\delta_+\alpha_- + \bar\delta_- \alpha_+).\]
 The commutation relation \eqref{eq:commutationrelation} becomes
 \[ \bar\delta_+\alpha_- + \bar\delta_- \alpha_+ = 0, \]
 where $\bar\delta_\pm$ are the components of the Bismut connection defined in \eqref{eq:partialBismut}.
\end{example}


\section{Stability and the Kobayashi-Hitchin correspondence}\label{Stability-KH}

In this section, we introduce the notions of $\alpha$-Hermitian-Einstein metric and $\alpha$-stability for $I_\pm$-holomorphic and generalized holomorphic bundles.
Our choice of $\alpha$-Hermitian-Einstein equation is motivated in section \S \ref{subsect:Kahlercohiggs}, where we consider generalized holomorphic bundles on K\"ahler
manifolds endowed with the trivial generalised K\"ahler structure. Moreover, some of the properties of $\alpha$-stable bundles are discussed in sections \ref{sect:biholbundles}
and \ref{generalizedholbun}. 

\subsection{The case of K\"ahler manifolds}\label{subsect:Kahlercohiggs}
Let $(M,g,I)$ be a K\"ahler manifold and consider the trivial generalized K\"ahler structure $(\JJ_I,\JJ_\omega)$ defined in example \ref{ex:Kahler}, whose
associated bi-Hermitian structure is $(g, I_+ = I_-= I)$. As we have seen in example \ref{co-Higgs}, all $\JJ_I$-holomorphic bundles correspond 
in this case to co-Higgs bundles, with respect to the complex structure $I$, and are therefore given by the following data:
a complex vector bundle $V$ on $M$ with 
\begin{itemize}
\item
a holomorphic structure $\bar\partial : C^\infty(V) \to C^\infty(V \tensor T^{0,1}M)$;
\item
a Higgs field $\psi \in C^\infty(\End(V) \tensor T_{1,0} M)$ such that $\bar\partial\psi = 0$ and $\psi \wedge \psi = 0$.
\end{itemize}

On the other hand, by Proposition \ref{prop:genholoisbiholo}, we know the co-Higgs bundle $(V,\bar\partial,\psi)$ also corresponds to an $I_\pm$-holomorphic bundle
$(V,\bar\partial_+,\bar\partial_-)$. Note that here 
\[\bar\partial_\pm: C^\infty(V) \to C^\infty(V \tensor T^{0,1}M)\] 
are just holomorphic structures on $V$ with respect to the complex structure $I$ since $I_\pm = I$. 
In fact, these $I_\pm$-holomorphic structures $\bar\partial_\pm$ can be constructed in terms of the co-Higgs data as follows. 

For any $\alpha \in (0,1)$, define
\begin{equation}\label{eq:inversejalpha}
\mbox{$\bar\partial_- = \bar\partial - \alpha\varphi$ \, and \, $\bar\partial_+ = \bar\partial + (1-\alpha)\varphi$ \, with \, $\varphi = \displaystyle \sqrt{\frac{1}{\alpha(\alpha-1)}}\omega \psi$.}
\end{equation}
Then, $\bar\partial_\pm$ are both derivation. Furthermore, $\bar\partial \varphi = 0$ because $d\omega = 0$ (since $g$ is K\"ahler) and $\bar\partial \psi = 0$. 
Therefore,
\[ \bar\partial_-^2 = \bar\partial^2 + \alpha^2\varphi \wedge \varphi = 0,\]
where $\varphi \wedge \varphi = 0$ comes from the fact that $\psi \wedge \psi = 0$. Since $\bar\partial_+ = \bar\partial_- + \varphi$, we also have $\bar\partial_+^2 = 0$, implying that
$(\bar\partial_+,\bar\partial_-)$ is an $I_\pm$-holomorphic structure on $V$. However, 
\[ [\bar\partial_+, \bar\partial_-] = [\bar\partial_- + \varphi, \bar\partial_-] = 2(\bar\partial_-)^2 + \bar\partial_- \varphi = 2(\bar\partial_-)^2 + \bar\partial \varphi - \varphi \wedge \varphi = 0, \]
so that $(\bar\partial_+,\bar\partial_-)$ satisfies the commutation relation \eqref{eq:commutationrelation}. In other words, $(\bar\partial_+,\bar\partial_-)$ corresponds to a
$\JJ_I$-holomorphic bundle.

Recall that a co-Higgs bundle $(V, \bar\partial,\psi)$ is called stable if the holomorphic bundle $(V, \bar\partial)$ satisfies the usual slope-stability condition 
for $\psi$-invariant coherent subsheaves \cite{Hitchin11,Rayan11}. Conversely, Hitchin suggests the following Yang-Mills-type equations for co-Higgs bundles \cite{Hitchin11}.
Fix a Hermitian metric $h$ on $V$ and let $F$ be the curvature of the Chern connection $\nabla^C = \partial + \bar\partial$ corresponding to $\bar\partial$. The co-Higgs bundle is 
then (poly)stable if $F$ and $\psi$ satisfy the equation:
\begin{equation}\label{eq:coHiggsHE}
 (F + [\varphi_0, \varphi_0^*]) \wedge \omega^{n-1} = (n-1)!\lambda \id_V \vol_g, \text{ where } \varphi_0 = \sqrt{-1}\omega\psi.
\end{equation}
These equations are the co-Higgs counterpart of Simpson's Yang-Mills equations for Higgs bundles on K\"ahler manifolds \cite{Simpson88}.
Also note that equation \eqref{eq:coHiggsHE} reduces to the Hermitian-Einstein equations for holomorphic vector bundles when $\psi = 0$.
Given the relationship we established between co-Higgs bundles and generalized holomorphic bundles at the beginning of the section, 
equation \eqref{eq:coHiggsHE} may give us a hint at possible Yang-Mills-type equations for generalized holomorphic bundles on generalized K\"ahler manifolds.

Let $\nabla^C_\pm$ be the Chern connections on $(V, h)$ corresponding to $\bar\partial_\pm$. Their curvatures are  
\begin{equation*}
 \begin{split}
  F_+ & = F + (1-\alpha) \partial \varphi - (1-\alpha)\bar\partial \varphi^* - (1-\alpha)^2 (\varphi \varphi^* + \varphi^*\varphi),\\
  F_- & = F -\alpha \partial \varphi +\alpha \bar\partial \varphi^* - \alpha^2 (\varphi \varphi^* + \varphi^*\varphi),
 \end{split}
\end{equation*}
respectively, giving us
\begin{equation}\label{eq:alphaHEleftside}
 \alpha F_+ +(1-\alpha)F_- = F - \alpha(1-\alpha)(\varphi\varphi^*+\varphi^*\varphi) = F + [\varphi_0, \varphi^*_0].
\end{equation}

\noindent
Equations \eqref{eq:coHiggsHE} and \eqref{eq:alphaHEleftside} then suggest the following Hermitian-Einstein equation for generalized holomorphic bundles on generalized K\"ahler manifolds:
\[\sqrt{-1}\left(\alpha F_+\wedge \omega_+^{n-1} + (1-\alpha) F_-\wedge \omega_-^{n-1}\right) = (n-1)!\lambda \id_V \vol_g.\]
This is the equation we will adopt in section \ref{sect:biholbundles} as our $\alpha$-Hermitian-Einstein equation \eqref{eq:alphaHE}.

\begin{remark}
Equations \eqref{eq:alphaHEleftside} and \eqref{eq:alphaHE} were derived by Hitchin for $\alpha = 1/2$ in \cite{Hitchin11}. 
However, the choice of $\alpha$ has an impact on the set of solutions of the equation (see sections \S \ref{sect:biholbundles} and \S \ref{subsect:rank2onhopf}). 
We therefore study the equation for all $\alpha \in (0,1)$.
\end{remark}

\subsection{$I_\pm$-holomorphic vector bundles}\label{sect:biholbundles}

Let $(V,\bar\partial_+, \bar\partial_-)$ be an $I_\pm$-holomorphic vector bundle on $M$ and $h$ be a Hermitian metric on $V$. Let $F_\pm$ be the curvatures of the Chern connections $\nabla^C_{\pm}$ on $V$ corresponding to the $I_\pm$-holomorphic structures $\bar\partial_\pm$. Fix $\alpha \in (0,1)$. 

We define the \emph{$\alpha$-Hermitian-Einstein equation} as:
\begin{equation}\label{eq:alphaHE}
 \sqrt{-1}\left(\alpha F_+\wedge \omega_+^{n-1} + (1-\alpha) F_-\wedge \omega_-^{n-1}\right) = (n-1)!\lambda \id_V \vol_g,
\end{equation}
where  $\lambda \in\RR$.
Note that this equation is a natural generalisation of the Hermitian-Einstein equation on a complex manifold endowed with a Gauduchon metric, to the bi-Hermitian setting; as in the complex case, we have:

\begin{defn}\label{defn:HermitianEinsteinconnections}
The Hermitian metric $h$ on $V$ is called \emph{$\alpha$-Hermitian-Einstein} if the Chern connections  $\nabla^C_\pm$ corresponding 
to the $I_\pm$-holomorphic structure $\bar\partial_\pm$ satisfy the $\alpha$-Hermitian-Einstein equation \eqref{eq:alphaHE}.
Alternatively, let $(\nabla_+, \nabla_-)$ be a pair of $h$-unitary connections on $V$. It is then called an \emph{$\alpha$-Hermitian-Einstein pair} if the curvatures $F_\pm$ of $\nabla_\pm$ are of type $(1,1)$ with respect to $I_\pm$, respectively, and satisfy equation \eqref{eq:alphaHE}. 
Hence, the Hermitian metric $h$ on $(V,\bar\partial_+, \bar\partial_-)$ is $\alpha$-Hermitian-Einstein if and only if the corresponding pair of Chern connections $(\nabla^C_+, \nabla^C_-)$ is an
$\alpha$-Hermitian-Einstein pair.
\end{defn}

In order to determine the $\alpha$-polystability of the $I_\pm$-holomorphic bundle $(V,\bar\partial_+,\bar\partial_-)$, 
we first need to define what we mean by $\alpha$-slope and by coherent subsheaf of $(V,\bar\partial_+,\bar\partial_-)$. 
Since we assumed the Riemannian metric $g$ to be Gauduchon with respect to both $I_+$ and $I_-$, 
we can associate to $V$ two degrees $\deg_\pm(V)$ and two slopes $\mu_\pm(V)$ in the standard way (see \cite{LubkeTeleman1995}, Definition 1.4.1):
 \[\deg_\pm(V) =\frac{\sqrt{-1}}{2\pi} \int_M \tr(F_\pm) \wedge \omega_\pm^{n-1}\] 
 and 
 \[\mu_\pm(V) = \frac{\deg_\pm(V)}{\rk V}.\] 
Note that $\deg_\pm(V)$ are independent of the choice of Hermitian metric $h$ on $V$ because the curvatures of Chern connections corresponding to different Hermitian metrics on $V$ differ by 
$\partial_\pm \bar\partial_\pm$-exact forms. Given these degrees and slopes, we now have:
 
\begin{defn}\label{defn:alphadegree}
We define the {\em $\alpha$-degree $\deg_\alpha(V)$} and the {\em $\alpha$-slope $\mu_\alpha(V)$} of $(V,\bar\partial_+,\bar\partial_-)$ as
\[ \deg_\alpha(V) := \alpha \deg_+(V)+ (1-\alpha) \deg_- (V)\]
and
\[\mu_\alpha(V) := \alpha \mu_+(V) + (1-\alpha)\mu_-(V),\]
respectively.
\end{defn}

\noindent
Furthermore, we define coherent subsheaves of $(V,\bar\partial_+,\bar\partial_-)$ as follows:
 \begin{defn} \label{defn:coherentsheaves}
  Let $\FFC_\pm$ be coherent subsheaves of $(V, \bar\partial_\pm)$, respectively. The pair $\FFC := (\FFC_+, \FFC_-)$ is said to be a \emph{coherent subsheaf} of $(V, \bar\partial_+, \bar\partial_-)$  if there exists analytic subsets $S_+$  and $S_-$ of $(M, I_+)$ and $(M, I_-)$, respectively, such that
  \begin{enumerate}
   \item $S := S_-\union S_+$ have codimension at least $2$;
   \item $\FFC_\pm|_{M \setminus S_\pm}$ are locally free and $\FFC_-|_{M\setminus S} = \FFC_+|_{M\setminus S}$.
  \end{enumerate}
The \emph{rank} of $\FFC$ is the rank of $\FFC_\pm|_{M \setminus S_\pm}$ and the \emph{$\alpha$-slope} of $\FFC$ is
  $$\mu_\alpha(\FFC) := \alpha \frac{\deg_+(\FFC_+)}{\rk(\FFC_+)} + (1-\alpha) \frac{\deg_-(\FFC_-)}{\rk(\FFC_-)}.$$
\end{defn}
 
\begin{remark*}
This notion of coherent subsheaf is motivated by the proof of theorem \ref{thm:setKHcorrespondence} (for details, see section \S \ref{subsect:coherentsheavesIpm}).
It would be interesting to understand the singularities of these sheaves with respect to the $I_\pm$-holomorphic structures, where they fail to be locally free, more concretely.
\end{remark*}

Let us now define $\alpha$-stability for $I_\pm$-holomorphic bundles.
\begin{defn}\label{defn:alphastable}
An $I_\pm$-holomorphic structure  $(\bar\partial_+, \bar\partial_-)$ on $V$ is called \emph{$\alpha$-stable} (resp.,  \emph{$\alpha$-semistable}), if, for any proper coherent subsheaf $\FFC$ of 
$(V,\bar\partial_+,\bar\partial_-)$, we have 
\[\mbox{$ \mu_\alpha(\FFC) < \mu_\alpha(V)$ (resp., $\mu_\alpha(\FFC) \leq \mu_\alpha(V)$}).  \]
Moreover,  $(\bar\partial_+, \bar\partial_-)$ is said to be \emph{$\alpha$-polystable} if it is a direct sum of $\alpha$-stable vector bundles with the same $\alpha$-slope.
\end{defn}

As in the complex case, there are several useful properties of $\alpha$-stability that follow directly from the definition:
\begin{enumerate}
\item
All $I_\pm$-holomorphic line bundles are $\alpha$-stable for all $\alpha \in (0,1)$.

\item
If $(V,I_+,I_-)$ does not admit a proper coherent subsheaf, then it is $\alpha$-stable for all $\alpha \in (0,1)$.
This can occur in two ways. Either $(V,\bar\partial_+)$ or $(V,\bar\partial_-)$ does not admit proper coherent subsheaves,
which is possible on non-projective manifolds (see example \ref{c2>0} on a Hopf surface). 
Or $(V,\bar\partial_+)$ and $(V,\bar\partial_-)$ both admit proper coherent subsheaves, but none agree away from an analytic set $S$ of codimension at least 2 
(see example \ref{c2=0}).

\item
On $I_\pm$-holomorphic bundles of rank-2, it is enough to check $\alpha$-stability only using sub-line bundles.

\item
Suppose $(g,I_+,I_-)$ is the bi-Hermitian structure associated to the trivial generalized K\"ahler structure on a K\"ahler manifold $(M,I,g)$ (see example \ref{ex:Kahler}),
so that $I_+ = I_- = I$. Therefore, $\deg_+(V) = \deg_-(V) = \deg(V)$, where $\deg(V)$ is the degree with respect to the initial K\"ahler structure $(g,I)$. Consequently,
$\deg_\alpha(V) = \deg(V)$ for all $\alpha \in (0,1)$, showing that $\alpha$-stability is independent of $\alpha$ in this case and coincides with the usual notion of stability
with respect to the K\"ahler metric $g$.
Nonetheless, when the bi-Hermitian structure does not come from a K\"ahler structure, $\alpha$-stability may depend on the choice of $\alpha$. This is indeed the case
on Hopf surfaces (see examples \ref{dependence-alpha:c2=0} and \ref{dependence-alpha:c2=1}).
\end{enumerate}

\noindent
Another property of classical stability that extends to the bi-Hermitian setting is the fact that $\alpha$-stable bundles are simple, where the notion of simple $I_\pm$-holomorphic bundle is defined as follows:
\begin{defn}\label{defn:simpleness}
 The $I_\pm$-holomorphic bundle $(V, \bar\partial_+, \bar\partial_-)$ is said to be {\em simple} if constant multiples of the identity are the only global endomorphism of $V$ that are holomorphic with respect to 
 both $\bar\partial_+$ and $\bar\partial_-$. More precisely, $(V, \bar\partial_+, \bar\partial_-)$ is simple if, for any $\Phi \in C^\infty(\End(V))$ such that
$\bar \partial_+ \Phi = 0$ and $\bar\partial_- \Phi = 0$, we have $\Phi = c \id_V$ for some $c \in \CC$.
\end{defn}

\noindent
Note that an $I_\pm$-holomorphic structure $(\bar\partial_+, \bar\partial_-)$ is trivially simple if one of $\bar\partial_\pm$ is simple as an $I_\pm$-holomorphic structure over $(M,I_\pm)$;
the other implication may, however, not be true.
Nevertheless, $\alpha$-stability implies simplicity for $I_\pm$-holomorphic bundles as in the classical case.
 \begin{lemma}\label{lemma:stabilityimpliessimplicity}
  If $(V, \bar\partial_+,\bar\partial_-)$ is stable, then it is simple.
 \end{lemma}
 \begin{proof}
  The slope $\mu_\bullet$ satisfies the following inequality
  \[ \mu_\bullet(S) \varleq \mu_\bullet(E) \varleq \mu_\bullet(Q)\]
  when $S,E$ and $Q$ are $I_\pm$-holomorphic bundles fitting into the short exact sequence
  \[ 0 \to S \to E \to Q \to 0, \]
 for $\mbox{$\bullet = +$ or $-$}$.
The $\alpha$-slope $\mu_\alpha$ therefore satisfies the same inequality for short exact sequences of $I_\pm$-holomorphic bundles. 
This inequality implies, as in the classical case, that stable $I_\pm$-holomorphic vector bundles  are simple.
 \end{proof}
  
 We end this section with the main result of the paper, which is a Kobayashi-Hitchin correspondence for $I_\pm$-holomorphic bundles.

  \begin{theorem}\label{thm:setKHcorrespondence}
 Let $(M,g,I_+,I_-)$ be a compact bi-Hermitian manifold such that $g$ is Gauduchon with respect to both $I_+$ and $I_-$, and $\vol_g = \frac{1}{n!} \omega_\pm^n$. 
 In addition, let $(V, \bar\partial_+, \bar\partial_-)$ be an $I_\pm$-holomorphic vector bundle on $M$. 
 In this case, $(V, \bar\partial_+, \bar\partial_-)$ admits an $\alpha$-Hermitian-Einstein metric if and only if it is $\alpha$-polystable, for any $\alpha \in (0,1)$.
 \end{theorem}
 \begin{proof}
 The proof is outlined in section \S \ref{sect:correspondenceforIpm}.
 \end{proof}

 \subsection{Generalized holomorphic bundles}\label{generalizedholbun}
 
 Because of Proposition \ref{prop:bihermitianholobundle}, the notions of $\alpha$-degree and $\alpha$-slope as well as $\alpha$-stability can be adapted to $\JJ$-holomorphic vector bundle, under Assumption \ref{assump:Gauduchon} that $g$ is Gauduchon with respect to both $I_+$ and $I_-$, and $\vol_g = \frac{1}{n!} \omega_\pm^n$.
\begin{defn}\label{defn:alphastablegenholo}
Let $\alpha \in (0,1)$. A $\JJ$-holomorphic vector bundle $(V, \bar\partial_+, \bar\partial_-)$ is called \emph{$\alpha$-(poly)stable} if the induced $I_\pm$-holomorphic structure $(\bar\partial_+, \bar\partial_-)$
on $V$ is $\alpha$-(poly)stable.
\end{defn}

From Theorem \ref{thm:setKHcorrespondence}, we thus obtain a Kobayashi-Hitchin correspondence for $\JJ$-holomorphic bundles on generalized K\"ahler manifolds:

\begin{corollary}\label{coro:generalizeKahlerKH}
Let $(M,\JJ,\JJ')$ be a compact generalized K\"ahler manifold whose associated bi-Hermitian structure $(g,I_+,I_-)$ satisfies Assumption \ref{assump:Gauduchon}
and let $(V, \bar\partial_+, \bar\partial_-)$ be a $\JJ$-holomorphic bundle on $M$. 
Then, $(V, \bar\partial_+, \bar\partial_-)$ admits an $\alpha$-Hermitian-Hermitian metric if and only if it is $\alpha$-polystable, for any $\alpha \in (0,1)$.
 \qed
\end{corollary}

 
\section{Hopf surfaces}\label{sect:Hopfexample}

In this section, we consider Hopf surfaces, which admit a natural bi-Hermitian structure $(g,I_+,I_-)$ that corresponds to an even generalized K\"ahler structure on the 
even-dimensional compact real Lie group $SU(2) \times S^1$ \cite{Gualtieri10}. This bi-Hermitian structure therefore satisfies Assumption \ref{assump:Gauduchon}.
We describe this structure in section \S \ref{subsect:bihermitian}. We then show that $I_\pm$-holomorphic and generalized holomorphic line bundles exist for this bi-Hermitian structure 
in section \S \ref{subsect:linebundleonHopf}. We finally consider rank-2 bundles in \S \ref{subsect:rank2onhopf}, giving examples of both $\alpha$-stable and $\alpha$-unstable
$I_\pm$-holomorphic and generalized holomorphic rank-2 bundles. In particular, we show that $\alpha$-stability can depend on the choice of $\alpha$ (see examples \ref{dependence-alpha:c2=0} and 
\ref{dependence-alpha:c2=1}). 

\subsection{Bi-Hermitian structure}\label{subsect:bihermitian}
There is a natural identification of $\CC^2$ with the quaternions $\bb H$:
\[(z_1, z_2) \mapsto z_1 + z_2 j.\]
Then $\tilde M = \bb H^\times$ is the Lie group of non-zero quarternions, with product and inverse given by
$$(z_1, z_2) \cdot (w_1, w_2) = (z_1w_1 - z_2 \bar w_2, z_1w_2 + z_2 \bar w_1) \text{ and } (z_1, z_2)^{-1} = \frac{1}{|z|^2}(\bar z_1, - z_2),$$
respectively, where $z = (z_1, z_2)$ and $|z|^2 = z_1\bar z_1 + z_2 \bar z_2$. In addition, the Hermitian metric
$$h: = \frac{1}{\pi|z|^2}(dz_1 d\bar z_1 + dz_2 d\bar z_2)$$
gives rise to a bi-invariant Riemannian metric $g$ on $\tilde M$.

The standard complex structure at $(1,0)$ induces the left- and right-invariant complex structures on $\tilde M$, which are respectively denoted $I_-$ and $I_+$. 
Note that $I_+$ corresponds to the standard complex structure on $\CC^2$. Referring to  \cite{Gualtieri10} (example 1.23), since $\tilde M$ is an even-dimensional real Lie group, $(g,I_+,I_-)$ 
is a bi-Hermitian structure on $\tilde M$ that corresponds to a generalized K\"ahler structure, and it descends to an even generalized K\"ahler structure $M$. 

In this section, we assemble a few technical facts about this bi-Hermitian structure that will be needed in sections \S \ref{subsect:linebundleonHopf} and \S \ref{subsect:rank2onhopf}.  
We first consider the complex structure $I_-$. The complexified tangent bundle of $M$ decomposes as  $TM \otimes \CC = T_{1,0}^-M \dsum T_{0,1}^-M$. 
A smooth frame for $T_{1,0}^-M$ is given by the left-invariant vector fields
\[ X_1 = z_1 \frac{\partial}{\partial z_1} + \bar z_2 \frac{\partial}{\partial \bar z_2}, X_2 = -\bar z_2\frac{\partial}{\partial \bar z_1} + z_1 \frac{\partial}{\partial z_2},\]
with dual frame 
\[ \alpha_1 = \frac{1}{|z|^2}(\bar z_1 dz_1 + z_2 d\bar z_2), \alpha_2 = \frac{1}{|z|^2}(-z_2d\bar z_1 + \bar z_1 dz_2)\]
 for $T^{1,0}_-M$, where $\alpha_i(X_j) = \delta_{ij}$ for $i,j = 1,2$.

\begin{lemma}\label{facts:I_-}
Facts about the complex structure $I_-$.
\begin{enumerate}
\item
$\partial_-(|z|^2) = |z|^2 \alpha_1$.

\item
$\bar\partial_-(\bar z_1/|z|^2) = \bar\partial_-(z_2/|z|^2) = 0$.

\item
$\partial_-\alpha_1 = 0$, $\partial_- \alpha_2 = - \alpha_1\wedge\alpha_2$, $\bar\partial_- \alpha_1 = \alpha_2 \wedge \bar \alpha_2$ and $\bar\partial_-\alpha_2 = \bar\alpha_1 \wedge \alpha_2$.

\item
The fundamental form of $g$ with respect to $I_-$ is
\[\omega_- = \frac{i}{2\pi}(\alpha_1 \wedge \bar\alpha_1 + \alpha_2 \wedge \bar\alpha_2)\]
and
\[d^c_- \omega_- = -\frac{1}{2\pi}(\bar\alpha_1 - \alpha_1) \wedge\alpha_2 \wedge \bar\alpha_2.\]
\end{enumerate}
\end{lemma}
\begin{proof}
We first note that the various differentials on functions can then be defined using the frames $\{X_1,X_2\}$ and $\{\alpha_1,\alpha_2\}$. For example,
$$\bar\partial_-f = \bar X_1(f) \bar\alpha_1 + \bar X_2(f) \bar\alpha_2$$
for any $f \in C^\infty(M)$. Hence,
\[ \partial_-(|z|^2) = X_1(|z|^2) \alpha_1 + X_2(|z|^2) \alpha_2 = |z|^2 \alpha_1,\]
giving (1); a similar computation gives (2).

We have $dz_1 = z_1\alpha_1 - z_2\bar \alpha_2$ and $dz_2 = z_2 \bar \alpha_1 + z_1\alpha_2$. Using equations $d(dz_1) = d(dz_2) = 0$ and their conjugates, we solve for  
$d\alpha_1$ and $d\alpha_2$ to obtain
$$d\alpha_1 = \alpha_2 \wedge \bar \alpha_2, d\alpha_2 = - \alpha_1 \wedge \alpha_2 + \bar\alpha_1\wedge \alpha_2.$$
It follows that $\partial_-\alpha_1 = 0$, $\partial_- \alpha_2 = - \alpha_1\wedge\alpha_2$, $\bar\partial_- \alpha_1 = \alpha_2 \wedge \bar \alpha_2$ and $\bar\partial_-\alpha_2 = \bar\alpha_1 \wedge \alpha_2$,
proving (3).

The metric $h$ can be written as
$$h = \frac{1}{\pi}(\alpha_1\bar\alpha_1 + \alpha_2\bar\alpha_2)$$
in the frame $\{ \alpha_1,\alpha_2\}$,
which implies that the corresponding K\"ahler form is
$$\omega_- = \frac{i}{2\pi}(\alpha_1 \wedge \bar\alpha_1 + \alpha_2 \wedge \bar\alpha_2).$$
Using property (2), a direct computation gives
$$d^c_- \omega_- = i(\bar\partial_- - \partial_-)\left(\frac{i}{2\pi}(\alpha_1 \wedge \bar\alpha_1 + \alpha_2 \wedge \bar\alpha_2)\right) = -\frac{1}{2\pi}(\bar\alpha_1 - \alpha_1) \wedge\alpha_2 \wedge \bar\alpha_2,$$
proving (4).
\end{proof}

We now turn to the complex structure $I_+$. A smooth frame for $T_{1,0}^+M$ is given by the right-invariant vector fields
\[Y_1 = z_1 \frac{\partial}{\partial z_1} + z_2 \frac{\partial}{\partial z_2}, Y_2 = -\bar z_2\frac{\partial}{\partial z_1} + \bar z_1 \frac{\partial}{\partial z_2},\]
with dual frame 
\[\beta_1 = \frac{1}{|z|^2}(\bar z_1dz_1 + \bar z_2dz_2), \beta_2 = \frac{1}{|z|^2}(-z_2dz_1 + z_1dz_2)\]
for $T^{1,0}_+$. Note that $\beta_i(Y_j) = \delta_{ij}$ for $i,j = 1,2$.

\begin{lemma}\label{facts:I_+}
Facts about the complex structure $I_+$.
\begin{enumerate}
\item
$\partial_+(|z|^2) = |z|^2 \beta_1$.

\item
$\partial_+\beta_1 = 0$, $\partial_+\beta_2 =\beta_1 \wedge \beta_2$, $\bar\partial_+ \beta_1 = - \beta_2\wedge \bar\beta_2$ and $\bar\partial_+\beta_2 = - \bar\beta_1\wedge \beta_2$.

\item
The fundamental form of $g$ with respect to $I_+$ is
\[\omega_+ = \frac{i}{2\pi}(\beta_1\wedge \bar \beta_1 + \beta_2 \wedge \bar\beta_2)\]
and
\[ d^c_+\omega_+ =  \frac{1}{2\pi}(\bar \beta_1 - \beta_1)\wedge \beta_2 \wedge \bar\beta_2.\]
\end{enumerate}
\end{lemma}
\begin{proof}
Properties (1) and (2) stem from computations similar to those outlined in the proof of lemma \ref{facts:I_-}.
Moreover, the metric $h$ is written as
$$h = \frac{1}{\pi}(\beta_1\bar\beta_1 + \beta_2\bar\beta_2)$$
in the frame $\{ \beta_1,\beta_2\}$, and the corresponding K\"ahler form is
$$\omega_+ = \frac{i}{2\pi}(\beta_1\wedge \bar \beta_1 + \beta_2 \wedge \bar\beta_2).$$
Using property (2), we obtain
$$d^c_+\omega_+ = i(\bar\partial_+ - \partial_+)\left(\frac{i}{2\pi}(\beta_1\wedge \bar \beta_1 + \beta_2 \wedge \bar\beta_2)\right) =  \frac{1}{2\pi}(\bar \beta_1 - \beta_1)\wedge \beta_2 \wedge \bar\beta_2,$$
proving (3).
\end{proof}

\begin{remark}\label{torsion-vol:hopf}
Referring to lemmas \ref{facts:I_-} and \ref{facts:I_+}, the torsion form $\gamma := d^c_+\omega_+ = -d^c_+\omega_-$ is given by
\begin{equation}\label{torsion:hopf}
\gamma = \frac{1}{2\pi}(\bar \beta_1 - \beta_1)\wedge \beta_2 \wedge \bar\beta_2 = \frac{1}{2\pi}(\bar\alpha_1 - \alpha_1) \wedge\alpha_2 \wedge \bar\alpha_2
\end{equation}
in terms of the frames $\{ \alpha_1,\alpha_2\}$ and $\{ \beta_1,\beta_2 \}$. Moreover, $\vol_g = \frac{1}{2}\omega_\pm^2$, implying the metric $g$ satisfies Assumption \ref{assump:Gauduchon}.
\end{remark}

We end with the following fact about the right- and left-invariant frames described above.
\begin{prop}\label{Bismut:Hopf}
The frames $\{\alpha_1,\alpha_2 \}$, $\{\beta_1,\beta_2\}$ and their conjugates are Bismut flat. That is,
\[ \nabla^- \alpha_i = \nabla^-\bar\alpha_i = \nabla^+ \beta_i = \nabla^+\bar\beta_i = 0,\]
$i = 1,2$, where $\nabla^\pm$ are the Bismut connections on $(M,g,I_\pm)$. This means, in particular, that 
\[ \bar\delta_+ \bar\alpha_i = \bar\delta_- \bar\beta_i = 0,\]
where $\bar\delta_\pm$ are the operators defined in \eqref{eq:Bismutextension}.
\end{prop}
\begin{proof}
Let us first show that  
\[ \nabla^- X_k = \nabla^-\bar X_k = \nabla^+ Y_k = \nabla^+\bar Y_k = 0,\]
$j = 1,2$; the result will then follow since $\{\alpha_1,\alpha_2 \}$, $\{\beta_1,\beta_2\}$ are the dual frames of $\{X_1,X_2 \}$, $\{Y_1,Y_2\}$, respectively. 
We use the description  \eqref{Bismut} of the Bismut connection in terms of the Courant-Dorfman bracket to compute these derivatives. First, recall that
$$\ell_- = \Span_\CC\left\{\mf X_1 = X_1 - \frac{1}{2\pi}\bar\alpha_1, \mf X_2 = X_2 - \frac{1}{2\pi} \bar\alpha_2\right\}$$
and
$$\ell_+ = \Span_\CC\left\{\mf Y_1 = Y_1 + \frac{1}{2\pi}\bar\beta_1, \mf Y_2 = Y_2 + \frac{1}{2\pi} \bar\beta_2\right\}$$
by \eqref{ellpm}.
Since left-(right-)invariant vector fields generate right (left) actions, we have
$$\LLC_{X_j} \beta_k = \LLC_{X_j} \bar\beta_k = \LLC_{Y_j} \alpha_k = \LLC_{Y_j}\bar\alpha_k = 0 \text{ and } [X_j, Y_k] = [X_j, \bar Y_k] = 0.$$
Also, note that since $\{X_1,X_2,\bar{X}_1,\bar{X}_2\}$ is a smooth frame for $TM \otimes \CC$, it is enough to check that
$$\nabla^+_{X_j} Y_k = \nabla^+_{\bar X_j} Y_k = 0 \text{ and } \nabla^+_{X_j} \bar Y_k = \nabla^+_{\bar X_j} \bar Y_k = 0,$$
for all $j,k$, in order to show that $\nabla^+ Y_k = \nabla^+\bar Y_k = 0$ for all $k$.
Referring to \eqref{Bismut}, we have
\[ \mbox{$\nabla^+_{X_j} Y_k =  a\left( (\mf X_j * \mf Y_k)^+ \right)$ and $\nabla^+_{\bar X_j} Y_k =  a\left( (\bar{\mf X}_j * \mf Y_k)^+ \right)$},\]
as well as
\[ \mbox{$\nabla^+_{X_j} \bar Y_k =  a\left( (\mf X_j * \bar{\mf Y}_k)^+ \right)$ and $\nabla^+_{\bar X_j} \bar Y_k =  a\left( (\bar{\mf X}_j * \bar{\mf Y}_k)^+ \right)$}.\]
Now, by definition,
$$\mf X_1 * \mf Y_1 = [X_1, Y_1] + \LLC_{X_1}\left(\frac{1}{2\pi}\bar\beta_1\right) + \iota_{Y_1}\left(\frac{1}{2\pi} d\bar\alpha_1\ - \iota_{X_1} \gamma \right).$$
Since
\[ [X_1, Y_1] = \LLC_{X_1}\left(\frac{1}{2\pi}\bar\beta_1\right) = 0\]
and 
$$d\left(\frac{1}{2\pi} \bar \alpha_1\right) - \iota_{X_1} \gamma = \frac{1}{2\pi} \bar \alpha_2 \wedge \alpha_2 - \iota_{X_1}\left(\frac{1}{2\pi} (\bar\alpha_1 - \alpha_1)\wedge \alpha_2 \wedge \bar\alpha_2\right) = 0,$$
we obtain $\mf X_1 * \mf Y_1 = 0$.
Similar computations give 
\[ \mf X_j *\mf Y_k = \bar{\mf X}_j * {\mf Y}_k = \mf X_j * \bar{\mf Y}_k = \bar{\mf X}_j * \bar{\mf Y}_k = 0 \] 
for all $j,k$, implying that $\nabla^+_{X_j} Y_k = \nabla^+_{\bar X_j} Y_k = 0$ and $\nabla^+_{X_j} \bar Y_k = \nabla^+_{\bar X_j} \bar Y_k = 0$ for all $j,k$. 
One checks that $\nabla^- X_k = \nabla^-\bar X_k = 0$ the same way using the frame $\{Y_1,Y_2,\bar{Y}_1,\bar{Y}_2\}$.
\end{proof}

\subsection{Holomorphic line bundles}\label{subsect:linebundleonHopf}
On a Hopf surface, all line bundles are flat, since $H^2(M,\ZZ) = 0$, and are therefore topologically trivial. Nonetheless, there exist many holomorphic structures on the trivial line bundle $M \times \CC$.
Indeed, these holomorphic structures are parametrised by $\rm{Pic}(M) = \CC^*$, up to isomorphism, and can be constructed via factors of automorphic as follows. 
For any $\eta \in \CC^*$, define an action on $\tilde M \times \CC$ by $(z,v) \mapsto (\tau z, \eta v)$. The quotient 
\[ L_\eta := \tilde M \times \CC/ (z,v) \sim (\tau z, \eta v)\]
of $\tilde M \times \CC$ by this action is then a flat complex line bundle on $M$. 
Note that $L_\eta \otimes L_\xi = L_{\eta\xi}$ and $(L_\eta)^* = L_{\eta^{-1}}$.
Moreover, $L_\eta$ admits a holomorphic structure locally given by the trivial connection with respect to any complex structure on $M$. Here are some facts about this holomorphic structure.

\begin{prop}\label{line-bundles}
Let $u$ denote the constant section $u(z) = (z, 1)$ on $M \times \CC$. Then, $L_\eta$ is isomorphic to $M \times \CC$ endowed with the $I_\pm$-holomorphic structure $\bar\partial_\pm$
given by
\[ \bar\partial_+(u) = c\bar\beta_1 \otimes u\]
and 
\[ \bar\partial_-(u) = c\bar\alpha_1 \otimes u,\]
where $c = \frac{\ln \eta}{2\ln \tau}$. Moreover, 
\begin{equation}\label{def:deg_pm} 
\deg_\pm L_\eta := \frac{i}{4\pi\vol(M)}{\int_M F_\pm \wedge \omega_\pm} = \pm \frac{\ln |\eta|}{\ln \tau},
\end{equation}
where $\vol(M) = \int_M \vol_g$ and $F_\pm$ are the curvatures of the Chern connections of $\bar\partial_\pm$ with respect to any Hermitian metric on $M \times \CC$.
\end{prop}
\begin{proof}
Note that $\sigma(z) = e^{c\ln|z|^2}$ defines a  smooth global section of $L_\eta$ (because $\sigma(\tau z) = \eta \sigma(z)$).
Since $\sigma$ is nowhere vanishing, it induces a global gauge transformation from $L_\eta$ to $M \times \CC$ that takes the
natural $I_\pm$-holomorphic structures on $L_\eta$ to the $I_\pm$-holomorphic structures $\bar\partial_\pm$ on $M \times \CC$ given by 
$\bar\partial_\pm(u)  = \bar\partial_\pm \sigma \cdot \sigma^{-1} \otimes u$ . Explicitly, since 
\[ \bar\partial_\pm \sigma = c \frac{\bar\partial_\pm (|z|^2)}{|z|^2}  \cdot \sigma, \] 
we have $\bar\partial_+(u) = c\bar\beta_1 \otimes u$ and $\bar\partial_-(u) = c\bar\alpha_1 \otimes u$ by lemmas \ref{facts:I_+} and \ref{facts:I_-}.

Let us now compute $\deg_\pm L_\eta$. Consider the standard Hermitian metric $h$ on $M \times \CC$, which is given by $h(z,v) = |v|$ . Then, $u$ is a Hermitian frame of $M \times \CC$,
implying that the Chern connections of $\bar\partial_\pm$ with respect to $h$ are given by the matrices $\Theta_+ = - \bar c \beta_1 + c\bar\beta_1$
and $\Theta_- = - \bar c \alpha_1 + c\bar\alpha_1$, respectively, in this frame. Consequently, the curvatures matrices of the Chern connections are
\[ F_+ = d\Theta_+ = \frac{\ln|\eta|}{\ln \tau} \beta_2\wedge\bar\beta_2\]
and 
\[F_- = d\Theta_- = -\frac{\ln|\eta|}{\ln\tau}\alpha_2 \wedge \bar\alpha_2.\]
Referring to lemmas \ref{facts:I_+} and \ref{facts:I_-} and remark \ref{torsion-vol:hopf}, we see that
\[ \frac{i}{4\pi}F_\pm\wedge \omega_\pm = \pm \frac{\ln|\eta|}{2\ln\tau} \omega_\pm \wedge \omega_\pm = \pm \frac{\ln|\eta|}{\ln\tau} \vol_g, \] 
implying that
\[\deg_\pm L_\eta = \frac{i}{4\pi\vol(M)}\int_M F_\pm \wedge \omega_\pm = \pm \frac{\ln|\eta|}{\ln \tau}.\]
Note that since the curvatures of Chern connections corresponding to different Hermitian metrics on $M \times \CC$ differ by $\partial_\pm \bar\partial_\pm$-exact forms, 
the expression $\int_M F_\pm \wedge \omega_\pm$ is in fact independent of $h$.
\end{proof}

\begin{remark}\label{O(m)}
Note that $M$ admits a natural projection $\pi$ onto $\PP^1$ given by $\pi(z_1,z_2) = [z_1:z_2]$, where $[z_1:z_2]$ are homogeneous coordinates on $\PP^1$.
This projection is holomorphic with respect to $I_+$. Consequently, if $\mathcal{O}_{\PP^1}(m)$ denotes the holomorphic line bundle on $\PP^1$ of degree $m$, 
its pullback $\pi^*\mathcal{O}_{\PP^1}(m)$ is an $I_+$-holomorphic line bundle on $(M,I_+)$.
Let $\mathcal{O}_\pm$ denote the sheaf of $I_\pm$-holomorphic functions on $(M,I_\pm)$. We set 
\[ \mathcal{O}_+(m) := \pi^*\mathcal{O}_{\PP^1}(m) \]
for all $m \in \ZZ$. Then, $\mathcal{O}_+(m)$ admits global holomorphic section if and only if $m \geq 0$, in which case these sections are homogeneous polynomials in $z_1$ and $z_2$ of degree $m$.
Let $p$ be such a polynomial. Then, $p(\tau z)  = \tau^m p(z)$, implying that $p$ is global holomorphic section of $L_{\tau^m}$. Thus, $\mathcal{O}_+(m) = L_{\tau^m}$
if $m \geq 0$. Nonetheless, $\mathcal{O}_+(-m) = \mathcal{O}_+(m)^* = L_{\tau^{-m}}$, implying that 
\[\mathcal{O}_+(m) = L_{\tau^m}\] 
for all $m \in \ZZ$.

Consider the inverse map $\iota: M \rightarrow M, (z_1,z_2) \mapsto (z_1,z_2)^{-1}$,
which is a biholomorphic map from $(M,I_-)$ to $(M,I_+)$. We set
\[ \mathcal{O}_-(m) := \iota^*\mathcal{O}_+(m) = L_{\tau^{-m}}\]
for all $m \in \ZZ$. The formula \eqref{def:deg_pm} for $\deg_\pm$ was therefore normalised to ensure that 
\[ \deg_\pm \mathcal{O}_\pm(m) = m.\] 
We finish with the observation that the only holomorphic line bundles on $(M,I_\pm)$ that have global holomorphic sections are $\mathcal{O}_\pm(m)$ with $m \geq 0$. 
Furthermore, if $L_\pm$ is any holomorphic line bundle on $(M,I_\pm)$, then $h^1(M,L_\pm) = 0$ unless $L_\pm = \mathcal{O}_\pm(m)$ with $m \neq -1$. 
For more facts about line bundles on Hopf surfaces, we refer the reader to \cite{Braam-Hurtubise,Moraru}.
\end{remark}

Clearly, any pair $(\bar\partial_+, \bar\partial_-)$ defines an $I_\pm$-holomorphic structure on $M \times \CC$. In fact, any such pair satisfies the commutation relation 
\eqref{eq:commutationrelation}, implying that it corresponds to a $\JJ$-holomorphic structure. 

\begin{prop}
 Any pair $(\bar\partial_+, \bar\partial_-)$ given above defines a $\JJ$-holomorphic vector bundle on $M$.
\end{prop}
\begin{proof}
Suppose that $\bar\partial_+$ and $\bar\partial_-$ are given by $\bar\partial_+(u) = c_+\bar\beta_1 \otimes u$ and 
$\bar\partial_-(u) = c_-\bar\beta_1 \otimes u$, respectively, where $u(z) = (z,1)$ is the constant section 
on $M \times \CC$. Then, referring to example \ref{exple:linebundlecommuteoperator}, we just need to check that $\bar\delta_+ \bar\alpha_1 + \bar\delta_- \bar\beta_1 = 0$. 
But this follows directly from proposition \ref{Bismut:Hopf}.
\end{proof}

\subsection{Rank-$2$ bundles}\label{subsect:rank2onhopf}
In this section, we discuss the existence of $\alpha$-stable rank-2 vector bundles on the Hopf surface $M$, 
endowed with the generalized K\"ahler structure described in section \S \ref{subsect:bihermitian}.  
Let $V$ be a fixed smooth complex rank-2 vector bundle on $M$. Then, $c_1(V) = 0$ since $H^2(M,\mathbb{Z}) = 0$. Set $c_2(V) = c_2$. 
Note that $V$ admits holomorphic structures if and only if $c_2 \geq 0$ (see \cite{Braam-Hurtubise, Moraru}).

\begin{example}\label{c2=0}
Assume $c_2 = 0$ so that $V$ is the trivial rank-2 vector bundle $M \times \CC^2$ on $M$. Let $\{e_1,e_2\}$ be the standard smooth global frame of $V$ with
$e_1(z) = (z, (1,0))$ and $e_2(z) = (z,(0,1))$ for all $z \in M$.
We endow $V$ with the following 
$I_\pm$-holomorphic structure $(\bar\partial_+,\bar\partial_-)$.  We describe $\bar\partial_+$ in terms of the frame $\{e_1,e_2\}$: 
\[ \mbox{$\bar\partial_+(e_1) = \eta\bar\beta_1 \otimes e_1$ and $\bar\partial_+(e_2) = \xi\bar\beta_1 \otimes e_2$} \]
for some $\eta,\xi \in \CC$ with $\eta \neq \xi$. Recall that a smooth sub-line bundle $L$ of $V$ is said to be a holomorphic sub-line bundle of $(V,\bar\partial_+)$ if $\bar\partial_+$ maps 
$C^\infty(L)$ to $C^\infty(L \otimes T^{1,0}_+M)$. An easy calculation shows that $(V,\bar\partial_+)$ has only two holomorphic sub-line bundles $L_1$ and $L_2$, namely, the sub-line
bundles of $V$ spanned by $e_1$ and $e_2$, respectively (which stems from the fact that we have assumed  $\eta \neq \xi$).

Let us now consider another smooth global frame of $V$, namely, $\{ f_1 = e_1 + e_2,f_2 = e_1 - e_2 \}$. The holomorphic structure $\bar\partial_-$ is described in terms of this frame as follows:
\[ \mbox{$\bar\partial_-(f_1) = a\bar\alpha_1 \otimes f_1$ and $\bar\partial_-(f_2) = b\bar\alpha_1 \otimes f_2$}\]
for some $a,b \in \CC$ with $a \neq b$.
We then have
\[ \bar\partial_-(e_1) = \frac{1}{2} \bar\alpha_1 \otimes ((a+b)e_1 + (a-b)e_2)\]
and 
\[ \bar\partial_-(e_2) = \frac{1}{2}\bar\alpha_1 \otimes ((a-b)e_1 + (a+b)e_2),\]
implying that $L_1$ and $L_2$ are not holomorphic sub-line bundles of $(V,\bar\partial_-)$. Consequently, $(V,\bar\partial_+,\bar\partial_+)$ does not have an $I_\pm$-holomorphic sub-line bundle. This means that $(V,\bar\partial_+,\bar\partial_+)$ is trivially $\alpha$-stable, for any $\alpha \in (0,1)$.
\end{example}

\begin{example}\label{c2>0}
Suppose $c_2 > 0$. For any complex structure $I$ on $M$, there then exist $I$-holomorphic rank-2 bundles that do not admit coherent sub-sheaves (see \cite{Braam-Hurtubise, Moraru}).
Fix an $I_+$-holomorphic structure $\bar\partial_+$ on $V$ such that $(V,\bar\partial_+)$ does not admit proper coherent sub-sheaves. 
Therefore, given any $I_-$-holomorphic structure $\bar\partial_-$ on $V$, the $I_\pm$-holomorphic bundle $(V,\bar\partial_+,\bar\partial_-)$ 
has no proper coherent sub-sheaves, implying that it is $\alpha$-stable for any $\alpha \in (0,1)$. One can thus construct many examples of $\alpha$-stable $I_\pm$-holomorphic bundle this way. 
Nonetheless, there also exist $\alpha$-stable $I_\pm$-holomorphic bundle that do admit proper coherent sub-sheaves, as illustrated by example \ref{dependence-alpha:c2=1}.
\end{example}

The above two examples give us the following result about the existence of $\alpha$-stable $I_\pm$-holomorphic structures on $M$:
\begin{prop}\label{alpha-stability:hopf}
For any $c_2 \geq 0$, there exist $I_\pm$-holomorphic structures on $V$ that are $\alpha$-stable  for all $\alpha \in (0,1)$.
\end{prop}
\begin{proof}
Follows directly from examples \ref{c2=0} and \ref{c2>0}.
\end{proof}

While the $I_\pm$-holomorphic bundles appearing in examples \ref{c2=0} and \ref{c2>0} are $\alpha$-stable for any $\alpha \in (0,1)$, 
there exist $I_\pm$-holomorphic bundle whose stability depends on the choice of $\alpha$, as illustrated by the next two examples.

\begin{example}\label{dependence-alpha:c2=0}
Assume $c_2=0$. We choose $I_\pm$-holomorphic structures $\bar\partial_\pm$ on $V$ as follows.
Let $\bar\partial_+$ be an $I_+$-holomorphic structure on $V$ such that $V_+ := (V,\bar\partial_+)$ is not isomorphic to a sum of two line bundles and is
given by a non-trivlal extension of the form
\[ 0 \rightarrow \mathcal{O}_+ \rightarrow V_+ \rightarrow \mathcal{O}_+(-m_+) \rightarrow 0, \]
with $m_+ \in \ZZ^{>0}$. This means in particular that $\mathcal{O}_+$ is the only $I_+$-holomorphic sub-line bundle of $V_+$, 
otherwise $V_+$ would be isomorphic to a sum of two line bundles, contradicting our assumption. 
Furthermore, let $\bar\partial_-$ be an $I_-$-holomorphic structure on $V$ such that $V_- := (V,\bar\partial_-)$ 
is given by a non-trivlal extension of the form
\[ 0 \rightarrow \mathcal{O}_- \rightarrow V_- \rightarrow \mathcal{O}_-(m_-) \rightarrow 0, \]
with $m_- \in \ZZ^{\geq 2}$. We assume the images of $\mathcal{O}_\pm$ in $V_\pm$ coincide as smooth sub-line bundles of $V$
so that $\mathcal{L} = (\mathcal{O}_+,\mathcal{O}_-)$ is an $I_\pm$-holomorphic sub-line bundle of $(V,\bar\partial_+,\bar\partial_-)$;
note that $\mathcal{L}$ is the only $I_\pm$-holomorphic sub-line bundle of $(V,\bar\partial_+,\bar\partial_-)$ 
since $\mathcal{O}_+$ is the only $I_+$-holomorphic sub-line bundle of $V_+$. 
Let 
\[ \alpha_0 =  \frac{m_-}{(m_+ + m_-)}.\]
Since $m_\pm$ are both positive integers, we have $0 <  \alpha_0 < 1$.
If we choose $\alpha \in (0,   \alpha_0)$, then
\[ \mu_\alpha(\mathcal{L})  = 0 < \frac{1}{2} \left( \alpha(- m_+) + (1-\alpha) m_- \right) = \mu_\alpha(V),\]
implying that $(V,\bar\partial_+,\bar\partial_-)$ is $\alpha$-stable. However, if we pick $\alpha \in [ \alpha_0,1)$,
then $\mu_\alpha(\mathcal{L})  = 0 \geq \mu_\alpha (V)$ and $(V,\bar\partial_+,\bar\partial_-)$ is not $\alpha$-stable.
The stability of $(V,\bar\partial_+,\bar\partial_-)$ thus depends on $\alpha$. 
We will see, in example \ref{stable-gen-hol}, that some of these $I_\pm$-holomorphic
structure $(\bar\partial_+,\bar\partial_-)$ correspond to $\JJ$-holomorphic structures on $V$. 
\end{example}

\begin{example}\label{dependence-alpha:c2=1}
Assume $c_2 = 1$. Let us choose $I_\pm$-holomorphic structures $\bar\partial_\pm$ on $V$ such that
both $V_\pm:=(V,\bar\partial_\pm)$ have determinant $\mathcal{O}_\pm$. This means in particular that $\deg_\pm V_\pm = 0$ so that $\deg_\alpha V = 0$. 
Suppose that both $V_\pm$ are given by extensions of the form
\[ 0 \rightarrow L_\pm \rightarrow V_\pm \rightarrow L_\pm^* \otimes I_{p,\pm} \rightarrow 0,\]
where $L_\pm$ are holomorphic line bundles and $I_{p,\pm}$ is the ideal of a point; such bundles can either be stable or unstable as holomorphic bundles on $(M,I_\pm)$
(see Theorem 5.2.2 in \cite{Braam-Hurtubise}). Let us choose $V_+$ to be stable
and $V_-$ to be unstable. Hence, for any holomorphic sub-line bundle $N_+$ of $V_+$, we have
\[ \deg_+(N_+) < \deg_+ V_+ = 0.\]
In fact, $L_+$ can be chosen such that
\[\deg_+ N_+ \leq \deg_+ L_+ < 0\] 
for any other holomorphic sub-line bundle $N_+$ (see Proposition 3.3.4 in \cite{Braam-Hurtubise}). In addition, $V_-$ can be given by a non-trivial extension of the form
\[ 0 \rightarrow \mathcal{O}_-(m_-) \rightarrow V_-\rightarrow \mathcal{O}_-(-m_-) \otimes I_{p,-} \rightarrow 0\]
with $m_- \in \mathbb{Z}^{>0}$ because $h^2(\mathcal{O}_-(-m_-)^* \otimes \mathcal{O}_-(m_-)) = 0$ (see Corollary 10 in \cite{Friedman}); 
in this case, any holomorphic sub-line bundle of $V_-$ must be of the form $\mathcal{O}_-(r)$ with $r \leq m_-$.
Finally, we can assume that the images of $L_+$ and $\mathcal{O}_-(m_-)$ in $V_+$ and $V_-$, respectively, coincide as smooth sub-line bundles of $V$
so that $\mathcal{L} = (L_+,\mathcal{O}_-(m_-))$ is an $I_\pm$-holomorphic sub-line bundle of $(V,\bar\partial_+,\bar\partial_-)$ with
\[ \deg_\alpha \mathcal{L} = \alpha \deg_+ L_+ + (1-\alpha)m_- = \alpha(\deg_+ L_+  - m_-) + m_-.\]
Let 
\[  \alpha_0 = \frac{m_-}{(m_- - \deg_+ L_+)} .\]
Note that $(m_- - \deg_+ L_+)  > m_-> 0$ so that $0 < \alpha_0  < 1$.
If one chooses $\alpha \in [\alpha_0,1)$, then $\deg_\alpha \mathcal{L} \geq 0$,
implying that $(V,\bar\partial_+,\bar\partial_-)$ is not $\alpha$-stable. On the other hand, if one chooses $\alpha \in (0,\alpha_0)$, 
then for any $I_\pm$-holomorphic sub-line bundle $\mathcal{N} = (N_+,\mathcal{O}(r))$ of $V$, we have
\[ \deg_\alpha \mathcal{N} = \alpha \deg_+ N_+ + (1-\alpha)r \leq \alpha \deg_+ L_+ + (1-\alpha)m_- = \alpha(\deg_+ L_+  - m_-) + m_- < 0,\]
implying that $(V,\bar\partial_+,\bar\partial_-)$ is $\alpha$-stable. The stability of $(V,\bar\partial_+,\bar\partial_-)$ therefore depends on $\alpha$.
\end{example}

We end this section by considering the existence of $\alpha$-stable $\JJ$-holomorphic bundles. We show in particular that such bundles always exist in the following

\begin{example}\label{stable-gen-hol}
Let us show that some of the $I_\pm$-holomorphic structures described in example \ref{dependence-alpha:c2=0} satisfy the commutation relation \eqref{eq:commutationrelation}.
In this case, the underlying complex vector bundle $V$ is topologically trivial. Let $\{ e_1,e_2 \}$ be the smooth global frame on $M \times \CC^2$ described in example \ref{c2=0}
and let $L_1$ and $L_2$ be the sub-line bundles spanned by $e_1$ and $e_2$, respectively. We put the following $I_\pm$-holomorphic structures on $L_1$ and $L_2$. 
We choose the trivial $I_\pm$-holomorphic structure given by the trivial connection with respect to both $I_\pm$ for $L_1$. As for $L_2$, we choose $(\bar\partial_+^{L_2},\bar\partial_-^{L_2})$ given by
\[ \mbox{$ \bar\partial_+^{L_2}(e_2) = -\frac{1}{2} m_+(\bar\beta_1 \otimes e_2)$  and $\bar\partial_-^{L_2}(e_2) = \frac{1}{2} m_- (\bar\alpha_1 \otimes e_2)$}.\]
Note that  $\bar\partial_+^{L_2}$ and $\bar\partial_-^{L_2}$ give rise to the holomorphic bundles $\mathcal{O}_+(-m_+)$ and $\mathcal{O}_-(m_-)$, respectively.
Moreover, since $V_+$ and $V_-$ are given by non-trivial extensions, they correspond to non-zero elements $\varphi_+ \in H^1(M,\mathcal{O}_+(-m_+)^* \otimes \mathcal{O}_+)$ and
$\varphi_- \in H^1(M,\mathcal{O}_-(m_-)^* \otimes \mathcal{O}_-)$, respectively. Let us choose $\varphi_\pm$ such that
\[ \mbox{$\varphi_+(e_2) = -\frac{1}{2} m_+(\bar\beta_1 \otimes e_1)$ and $\varphi_-(e_2) = \frac{1}{2} m_-(\bar\alpha_1 \otimes e_1)$}.\]
A direct computation gives $\bar\partial_\pm \varphi_\pm = 0$, implying that $\varphi_\pm$ are holomorphic sections. Let us check it for $\varphi_+$. 
Since $\bar\beta_1$ is Bismut flat, we have 
\[ \bar\partial_+^{L_2^*\otimes \mathcal{O}_+} (\varphi_+) = -\frac{1}{2} m_+(\bar\beta_1 \otimes  \bar\partial_+^{L_2^*}(e_2^*) \otimes e_1) = 
-\frac{1}{4} m_+^2(\bar\beta_1 \wedge \bar\beta_1 \otimes e_2^* \otimes e_1) = 0.\]
We then choose the $I_\pm$-holomorphic structure $(\bar\partial_+,\bar\partial_-)$ on $V$ given by
\[ \bar\partial_\pm = 
\begin{pmatrix}
0 & \varphi_\pm \\
\medskip
0 & \bar\partial_\pm^{L_2}
\end{pmatrix}.\]
Then, one easily checks that $(\bar\partial_+,\bar\partial_-)$ satisfies the commutation relation \eqref{eq:commutationrelation}. 
Indeed, 
\[\bar\partial_+ \circ \bar\partial_-(e_1) + \bar\partial_- \circ \bar\partial_+(e_1) = 0\] 
since $\bar\partial_\pm(e_1) = 0$; moreover
\[ \bar\partial_+ \circ \bar\partial_-(e_2) + \bar\partial_- \circ \bar\partial_+(e_2)   = 
\bar\partial_+\left(\frac{1}{2} m_-\bar\alpha_1 \otimes (e_1 + e_2) \right) + \bar\partial_-\left(-\frac{1}{2} m_+\bar\beta_1 \otimes (e_1 + e_2) \right) \]
\[ = \frac{1}{4} m_+ m_-(\bar\alpha_1 \otimes  \bar\beta_1 \otimes e_1)  + \frac{1}{4} m_+ m_-(\bar\beta_1 \otimes  \bar\alpha_1 \otimes e_1) = 0.\]
Consequently, $(V,\bar\partial_+,\bar\partial_-)$ corresponds to a $\JJ$-holomorphic structure on $V$, which is $\alpha$-stable when 
$0 < \alpha < \frac{m_-}{(m_+ + m_-)} $.
\end{example}

\begin{remark}
It is shown in \cite{Gualtieri-Hu-Moraru} that $\JJ$-holomorphic bundles on a Hopf surface endowed with the generalized K\"ahler structure described in \S \ref{subsect:bihermitian} are topologically flat,
implying in particular that their underlying complex vector bundle $V$ has $c_2(V) = 0$. No $I_\pm$-holomorphic structures $(\bar\partial_+,\bar\partial_+)$ can then correspond to $\JJ$-holomorphic structures when $c_2>0$, implying that such structures do not satisfy the commutation relation \eqref{eq:commutationrelation}.
In addition, a direct computation shows that the $I_\pm$-holomorphic structures described in example \ref{c2=0} also do not satisfy the commutation relation  \eqref{eq:commutationrelation}. 
Consequently, not all $I_\pm$-holomorphic structures correspond to $\JJ$-holomorphic structures.
\end{remark}


\section{Proof of Theorem \ref{thm:setKHcorrespondence}}\label{sect:correspondenceforIpm}

Let $(M,g,I_+,I_-)$ be a compact bi-Hermitian manifold whose Riemannian metric $g$ satisfies Assumption \ref{assump:Gauduchon}. 
In this section, we prove Theorem \ref{thm:setKHcorrespondence}, 
which states that an $I_\pm$-holomorphic bundle on $M$ admits an $\alpha$-Hermitian-Einstein metric if and only if it is 
$\alpha$-polystable for any $\alpha \in (0,1)$. Our proof follows Chapters 2 and 3 of \cite{LubkeTeleman1995}.

Fix $\alpha \in (0,1)$. Let $V$ be a complex vector bundle on $M$ and $h$ be a Hermitian metric on $V$. 
Given an $I_\pm$-holomorphic structure $(\bar\partial_+,\bar\partial_-)$ on $V$, denote $F^h_\pm$ the curvatures of the
Chern connections $\nabla^C_\pm$ on $V$ corresponding to $h$. We then associate to $(V,\bar\partial_+,\bar\partial_-)$ two \emph{mean curvatures} $$K_\pm := \sqrt{-1} \Lambda_\pm (F^h_\pm)$$
and an \emph{$\alpha$-mean curvature} defined by
\begin{equation}\label{eq:alpha-mean-curv}
 K_\alpha := \alpha K_+ + (1-\alpha)K_-.
\end{equation}
The $\alpha$-Hermitian-Einstein equation \eqref{eq:alphaHE} can be rewritten in terms of the $\alpha$-mean curvature as
\begin{equation}\label{eq:alphaHE-new}
 K_\alpha  = \lambda \id_V.
\end{equation}
The version of the $\alpha$-Hermitian-Einstein equation we use throughout this section is equation \eqref{eq:alphaHE-new}.

The implications of Theorem \ref{thm:setKHcorrespondence} are proven in separate sections, namely \S  \ref{subsect:HEtostable} and \S \ref{subsect:stabletoHE}.
\subsection{$\alpha$-Hermitian-Einstein implies $\alpha$-polystability}\label{subsect:HEtostable}
In this section, we prove that any $I_\pm$-holomorphic line bundle  admits $\alpha$-Hermitian-Einstein metrics for all $\alpha \in (0,1)$ (see Corollary \ref{line-bundle:stable}). 
We also show that if an $I_\pm$-holomorphic vector bundle $(V,\bar\partial_+,\bar\partial_-)$ admits an $\alpha$-Hermitian-Einstein metric, then it is $\alpha$-polystable
(see Corollary \ref{coro:injectivityofbarK}).
Note that our proofs follow closely the presentation of Chapter 2 in \cite{LubkeTeleman1995}. 
We thus only provide explicit details where the proofs differ and refer the reader to  \cite{LubkeTeleman1995} for the rest. Before proving Corollaries \ref{line-bundle:stable}
and \ref{coro:injectivityofbarK}, we need to define some operators.

When the $I_\pm$-holomorphic structures $\bar\partial_\pm$ are fixed on $V$, the $\alpha$-Hermitian-Einstein equation \eqref{eq:alphaHE-new} is an equation of the Hermitian metric $h$ on $V$. In this case, we define
\[ P_{\pm,h} := \sqrt{-1} \Lambda_\pm\bar\partial_\pm \partial_\pm\] 
and 
\[ P^\alpha_h := \alpha P_{+,h} + (1-\alpha)P_{-,h},\]
where $\nabla^C_\pm = \partial_\pm + \bar\partial_\pm$ are the Chern connections corresponding to $h$. 
Here are some technical facts about the operator $P^\alpha_h$.
Following the arguments in \S 7.2 of \cite{LubkeTeleman1995}, the symbol of $P^\alpha_h$ is given by
\[ \sigma(P^\alpha_h)(x, u)(v) = \left(\alpha\left|u^{1,0}_+\right|^2 + (1-\alpha)\left|u^{1,0}_-\right|^2\right)\cdot v, \]
for $u \in T_x^*M$ and $v \in V_x$,
where $u_\pm^{1,0}$ denotes the $(1,0)$-part of $u$ under $I_\pm$. Thus, $P^\alpha_h$ is elliptic. Since $\sigma(P^\alpha_h)$ is self-adjoint, $\ind(P^\alpha_h) = 0$. 
Choosing local coordinates, $P^\alpha_h$ is a second order operator with negative definite leading coefficients.  
The maximum principle then induces the following analogue of Lemma 7.2.7 in \cite{LubkeTeleman1995}.
\begin{lemma}\label{lemma:PalphaonCinfty}
 The operator $P^\alpha_h : C^\infty(M) \to C^\infty(M)$ satisfies the following:
\begin{enumerate}
 \item $\ker(P^\alpha_h) \cong \CC$ consists of constant functions on $M$;
 \item $\displaystyle{\img(P^\alpha_h) = \left\{f: \int_M f \vol_g = 0 \right\}}$;
 \item $P^\alpha_h$ is self-adjoint.
\end{enumerate}
\qed
\end{lemma}

We can now prove the following analogue of Corollary 2.1.6 in \cite{LubkeTeleman1995}.
\begin{corollary}\label{line-bundle:stable}
$I_\pm$-holomorphic line bundles  admit $\alpha$-Hermitian-Einstein metrics for all $\alpha \in (0,1)$.
\end{corollary}
\begin{proof}
 Let $(V,\bar\partial_+,\bar\partial_-) $ be an $I_\pm$-holomorphic line bundle on $M$ and let $h_0$ be a Hermitian metric on $V$. Therefore, $K_\alpha^{h_0} = f id_V$ for some function $f: M \to \RR$.
 Define $h = e^k h_0$ for some $k : M \to \RR$. Then,
 \[ K_\alpha^{h} = \left(f - (n-1)!P_h^\alpha(k)\right)\vol_g.\]
Choose $\lambda \in \RR$ such that $\displaystyle{\int_M (f - \lambda) \vol_g = 0}$. There then exists $k \in C^\infty(M)$ such that 
\[ \displaystyle{P_h^\alpha(k) = \frac{f-\lambda}{(n-1)!}}, \] 
up to a constant in $\CC$. It follows that the corresponding Hermitian metric $h = e^k h_0$ satisfies the $\alpha$-Hermitian-Einstein equation \eqref{eq:alphaHE-new}.
\end{proof}

As for holomorphic vector bundles on complex manifolds, we have a notion of holomorphic section for $I_\pm$-holomorphic vector bundles.
\begin{defn}\label{defn:Ipm-hol-sect}
A section $s$ of $(V, \bar\partial_+, \bar\partial_-)$ is called \emph{$I_\pm$-holomorphic} if $\bar\partial_\pm s = 0$. 
\end{defn}

\noindent
A key ingredient in the proof of Corollary \ref{coro:injectivityofbarK} is the following vanishing theorem for $I_\pm$-holomorphic sections of $V$, which is similar to Theorem 2.2.1 in \cite{LubkeTeleman1995}.

\begin{theorem}\label{thm:vanishing}
If $\lambda < 0$, then $(V, \bar\partial_+, \bar\partial_-)$ has no global $I_\pm$-holomorphic sections. 
In addition, if $\lambda = 0$ and $\alpha \in (0,1)$, then every global $I_\pm$-holomorphic section  of $(V, \bar\partial_+, \bar\partial_-)$ is parallel with respect to both $\nabla^C_\pm$.
\end{theorem}
\begin{proof}
It can be shown that, for a global $I_\pm$-holomorphic section  $s$,
$$P^\alpha_h(h(s, s)) = \lambda |s|^2 - (\alpha|\partial_+s|^2 + (1-\alpha)|\partial_-s|^2).$$
For $\lambda < 0$, the maximum principle implies that $P^\alpha(h(s,s)) = 0$. Consequently, $|s|^2=0$ and $s = 0$. 
For $\lambda = 0$ and $\alpha \in (0,1)$, we see that $\partial_\pm s = 0$, which implies $\nabla^C_{\pm} s = 0$.
\end{proof}

Referring to Definition 1.1.16 and Proposition $1.1.17$ in \cite{LubkeTeleman1995}, we define the following analogue of an irreducible connection in the bi-Hermitian setting.
\begin{defn}\label{defn:irreduciblepairs}
 An pair $(\nabla_+, \nabla_-)$ of $h$-unitary connections on $V$ is called \emph{irreducible} if $V$ cannot be written as an $h$-orthogonal and $(\nabla_+, \nabla_-)$-parallel direct sum of nontrivial subbundles.
\end{defn}

\noindent
Using the vanishing theorem \ref{thm:vanishing}, the following is proved exactly like Theorem $2.3.2$ in \cite{LubkeTeleman1995}:
\begin{theorem}\label{thm:HEtostable}
 Let $(\nabla_+, \nabla_-)$ be an irreducible $\alpha$-Hermitian-Einstein pair of $h$-unitary connections on $V$. 
 Let $\bar\partial_\pm$ be the $(0,1)$-part of $\nabla_\pm$ with respect to the complex structure $I_\pm$.
 Then, $(\bar\partial_+, \bar\partial_-)$ is an $\alpha$-stable $I_\pm$-holomorphic structure  on $V$.
\end{theorem}
\begin{proof}
 Assume that $h$ satisfies the $\alpha$-Hermitian-Einstein equation. Then, 
 \begin{equation}\label{eq:HEconstant}
  \lambda(V) = \frac{2\pi}{(n-1)! \Vol_g(M)} \mu_\alpha(V).
 \end{equation}
Suppose that $\mc F$ is a coherent subsheaf of $(V, \bar\partial_+, \bar\partial_-)$. Let $r_V = \rk V$ and $r_{\mc F} = \rk \mc F$. 
The inclusion $\mc F \hookrightarrow V$ then induces the injective map $\det \mc F \to \wedge^{r_{\mc F}} V$, which defines a section $s$ of $\wedge^{r_E} V \tensor \det  \mc F^*$. 
By construction, $s$ is an $I_\pm$-holomorphic section with respect to the induced $I_\pm$-holomorphic structure.
Suppose $(\nabla_+, \nabla_-)$ is an $\alpha$-Herimitian-Einstein pair. Then the same argument using the vanishing theorem as in Theorem 2.3.2  of \cite{LubkeTeleman1995} shows that
 $$r_{\mc F} \lambda(V) - \lambda(\det \mc F) \vargeq 0 \Rightarrow \mu_\alpha(\mc F) \varleq \mu_\alpha(V).$$
 Since $(\nabla_+, \nabla_-)$ is irreducible, we see that $(V, \bar\partial_+, \bar\partial_-)$ is $\alpha$-stable.
\end{proof}

In general, we have the following result, which essentially corresponds to Theorem 2.3.2 in \cite{LubkeTeleman1995}.
\begin{corollary}\label{coro:injectivityofbarK}
 If $(V, \bar\partial_+, \bar\partial_-)$ admits an $\alpha$-Hermitian-Einstein metric, then it is $\alpha$-polystable.
\end{corollary}
\begin{proof}
The proof follows from Theorem \ref{thm:HEtostable} and the arguments appearing in the proof of the second half of Theorem 2.3.2 in \cite{LubkeTeleman1995}.
\end{proof}

\subsection{$\alpha$-polystability implies $\alpha$-Hermitian-Einstein}\label{subsect:stabletoHE}
It is enough to show that if $(V,\bar\partial_+,\bar\partial_-)$ is $\alpha$-stable, then it admits an $\alpha$-Hermitian-Einstein metric.
Our proof of this statement follows closely the line of argument of Chapter 3 in \cite{LubkeTeleman1995}. 
We give precise references and provide explicit details of how the proof is adapted to the bi-Hermitian setting throughout.
We begin the section by outlining a strategy for the proof. 

Let us fix the $I_\pm$-holomorphic structures $\bar\partial_\pm$ and a Hermitian metric $h_0$ on $V$. 
For any positive-definite Hermitian endomorphism $f \in \Herm^+(V, h_0)$, let $h := fh_0$ be the Hermitian metric defined by
$$h(s, t) := h_0(fs, t)$$
 for $s, t \in C^\infty(V)$.
Let $\nabla^C_\pm = \bar\partial_\pm + \partial_{\pm, 0}$ be the Chern connections of $\bar\partial_\pm$ with respect to $h_0$. Their curvatures are then
$$F_\pm^{h_0} = \bar\partial_\pm(\partial_{\pm,0} h_0 \cdot h_0^{-1}),$$
where we also use $h_0$ to denote the matrix of the metric $h_0$ with respect to a local holomorphic frame. 
Consequently, the curvatures of the Chern connections of $\bar\partial_\pm$ with respect to $h$ are
$$F_\pm^h = F_\pm^{h_0} + \bar\partial_\pm(f^{-1}\partial_{\pm,0} (f))$$
and, referring to \eqref{eq:alpha-mean-curv}, the $\alpha$-mean curvature $K_\alpha^h$ with respect to $h$ is given by
$$K_\alpha^h = K_\alpha^{h_0} + \sqrt{-1}(\alpha\Lambda_+\bar\partial_+(f^{-1}\partial_{+,0} (f)) + (1-\alpha)\Lambda_-\bar\partial_-(f^{-1}\partial_{-,0} (f))).$$
Set
\[ K_\alpha^0 := K_\alpha^{h_0} - \lambda \id_V, \] 
where $\lambda := \lambda(V)$ is given by \eqref{eq:HEconstant}. 
The $\alpha$-Hermitian-Einstein equation \eqref{eq:alphaHE-new} for $h = f h_0$ can then be expressed in terms of $h_0$ and $f$ as
\begin{equation}\label{eq:alphacontinuity0}
 K_\alpha^0 + \sqrt{-1}(\alpha\Lambda_+\bar\partial_+(f^{-1}\partial_{+,0} (f)) + (1-\alpha)\Lambda_-\bar\partial_-(f^{-1}\partial_{-,0} (f))) = 0.
\end{equation}

As explained in \cite{LubkeTeleman1995}, section \S 3.1, the \emph{continuity method} consists in solving equation \eqref{eq:alphacontinuity0} by considering the perturbed equation
\begin{equation}\label{eq:alphacontinuity1}
 L^\alpha_\varepsilon(f):= K_\alpha^0 + \sqrt{-1}(\alpha\Lambda_+\bar\partial_+(f^{-1}\partial_{+,0} (f)) + (1-\alpha)\Lambda_-\bar\partial_-(f^{-1}\partial_{-,0} (f))) + \varepsilon \log(f) = 0,
\end{equation}
for $\varepsilon \in [0,1]$.
Note that equation \eqref{eq:alphacontinuity0} is equivalent to $L^\alpha_0(f) = 0$. The existence of an $\alpha$-Hermitian-Einstein metric $h$ on $V$ is therefore equivalent to the existence of a solution 
$f \in \Herm^+(V, h_0)$ of $L^\alpha_0(f) = 0$. 
The rest of this section deals with proving the existence of such a solution using the continuity method when $(V,\bar\partial_+,\bar\partial_-)$ is $\alpha$-stable.

\subsubsection{Summary of the proof and notation}
In this subsection, we provide a summary of the proof as well as some of the notation that will be used in the remainder of the paper.

Suppose that $(V, \bar\partial_+, \bar\partial_-)$ is $\alpha$-stable. Consider the set
\[ J = \left\{\varepsilon \in [0,1] : \mbox{ there exists $f_\varepsilon \in \Herm^+(V, h_0)$ such that $L^\alpha_\varepsilon(f_\varepsilon) = 0$}\right\}.\]
Then, $J$ is non-empty. Indeed, we have:

\begin{lemma}
$1 \in J$, implying that $J \neq \emptyset$.
\end{lemma}
\begin{proof}
As in the proof of Lemma 3.2.1 in \cite{LubkeTeleman1995}, there exists a Hermitian metric $h_0$ on $V$ such that $\tr K_\alpha^0 = 0$ and $f_1 \in \Herm^+(V, h_0)$ with $L^\alpha_1(f_1) = 0$,
implying that $1 \in J$.
\end{proof}

\noindent
In fact, $J$ is both an open  and a closed subset of $(0,1]$ (see subsections \S \ref{subsect:openness} and \S \ref{subsect:closedness}, respectively).
Since $(0,1]$ is connected, this implies that $J = (0,1]$. There are now two possibilities:
\begin{itemize}
\item 
If $\lim_{\varepsilon \rightarrow 0} f_\epsilon = f_0$ exists for some $f_0 \in \Herm^+(V, h_0)$ such that $L^\alpha_0(f_0) = 0$, then
\[ h(s, t) := h_0(fs, t), \]
$s,t \in C^\infty(V)$, is an $\alpha$-Hermitian-Einstein metric on $(V,\bar\partial_+,\bar\partial_-)$ (see Corollary \ref{coro:L2boundgiveslimit}).

\item
If $\lim_{\varepsilon \rightarrow 0} f_\epsilon = f_0$ does not exist for some $f_0 \in \Herm^+(V, h_0)$ such that $L^\alpha_0(f_0) = 0$, then 
\begin{equation}\label{eq:limitDNE} 
\limsup\limits_{\varepsilon\rightarrow 0} \norm{\log f_\varepsilon}_{L^2} = \infty
\end{equation}
and $(V, \bar\partial_+, \bar\partial_-)$ is not $\alpha$-stable (by Corollary \ref{coro:L2boundgiveslimit} and Proposition \ref{prop:destabilizing}, respectively).
To be precise, the proof of Proposition \ref{prop:destabilizing} consists in showing that if \eqref{eq:limitDNE} holds,
then $(V, \bar\partial_+, \bar\partial_-)$ admits a {\em destablizing subsheaf} (see subsection \S \ref{subsect:coherentsheavesIpm} for details).
\end{itemize}

\noindent
Consequently, since we assumed $(V, \bar\partial_+, \bar\partial_-)$ to be $\alpha$-stable, it must admit an $\alpha$-Hermitian-Einstein metric, concluding the proof.

\begin{remark*}
We note that the proof outlined above does not require the $\alpha$-stability of $(V, \bar\partial_+, \bar\partial_-)$, but rather just the fact that $(V, \bar\partial_+, \bar\partial_-)$ is simple 
(which follows from the $\alpha$-stability by Lemma \ref{lemma:stabilityimpliessimplicity}), to prove that $J$ is a closed subset of $(0,1]$.
\end{remark*}

\noindent
{\bf Notation.} In the remainder of the paper, we drop the superscript $\alpha$ from the notation $L^\alpha_\varepsilon$. Moreover, we introduce the following short hand:
\begin{equation}\label{eq:shorthand}
\Lambda^\alpha \bar\partial(f^{-1}\partial f) := \alpha\Lambda_+\bar\partial_+(f^{-1}\partial_{+} (f)) + (1-\alpha)\Lambda_-\bar\partial_-(f^{-1}\partial_{-} (f))
\end{equation}
and use $\Lambda^\alpha \bar\partial(f^{-1} \partial_0 f)$ to denote the right-hand side of equation \eqref{eq:shorthand} with $\partial_\pm$ is replaced by $\partial_{\pm,0}$.
Furthermore, to simplify notation, we set
\[ P_\pm := P_{\pm, h_0} = \sqrt{-1} \Lambda_\pm\bar\partial_\pm \partial_{\pm,0}\]
and 
\[ P^\alpha:= P^\alpha_{h_0} = \alpha P_{+,h_0} + (1-\alpha)P_{-,h_0}.\]
Finally, for any $f \in \Herm^+(V, h_0)$, we let 
$$\Ad^{\pm\half}_f(\psi) := f^{\pm\half}\circ \psi \circ f^{\mp\half}$$
and define
$$d_\pm^f := \partial_{\pm, 0}^f + \bar\partial_\pm^f,$$
where
 \begin{equation}\label{eq:twisted-partial}
\mbox{ $\partial_{\pm,0}^f := \Ad^{-\half}_f \circ \partial_{\pm,0} \circ \Ad^\half_f$ and $\bar\partial_\pm^f := \Ad^\half_f \circ \bar\partial_\pm \circ \Ad^{-\half}$.}
 \end{equation}


\subsubsection{Openness of $J$}\label{subsect:openness}
We show that $J$ is an open subset of $(0,1]$ by using the Implicit Function Theorem. 
We follow the line of argument given in \cite{LubkeTeleman1995}, section \S 3.2.

Before proving that $J$ is an open subset of $(0,1]$  in Proposition \ref{prop:openness}, we need to establish a few technical results.
We begin with the following analogue of Lemma 3.2.3 in \cite{LubkeTeleman1995}:
\begin{lemma}\label{lemma:Lhat} 
 For any $f \in \Herm^+(V, h_0)$, we have $\hat L(\varepsilon, f) : = f\circ L_\varepsilon(f) \in \Herm(V, h_0)$.
\end{lemma}
\begin{proof}
We first note that  $K_\alpha^0$ can be written as $K_\alpha^0 = \alpha K_+^0 + (1-\alpha) K_-^0$,
where $K_\pm^0  := K^{h_0}_\pm - \lambda \id_V$. 
Consequently, $L_\varepsilon = \alpha L^+_\varepsilon + (1-\alpha) L^-_\varepsilon$ with
\[ L^\pm_\varepsilon(f) := K_\pm^0 + \sqrt{-1}\Lambda_\pm(\bar\partial_\pm(f^{-1}\circ \partial_{\pm, 0}(f))) + \varepsilon \log(f). \]
 It follows from \cite{LubkeTeleman1995}, Lemma 3.2.3, that $\hat L^\pm(\varepsilon, f) := f\circ L_\varepsilon^\pm(f) \in \Herm(V, h_0)$,
 implying that the convex combination $\hat L(\varepsilon, f) \in \Herm(V, h_0)$ as well.
\end{proof}

Similarly to Lemma 3.2.4 in \cite{LubkeTeleman1995}, we work in the Sobolev space $L^p_k\Herm(V, h_0)$ and have
\begin{lemma}\label{lemma:linearization} 
 The linearization 
 $$d_2\hat L(\varepsilon, f) : L^p_k\Herm(V, h_0) \to L^p_{k-2}\Herm(V, h_0)$$
 of $\hat L(\varepsilon, f)$ at $f$ is a second order elliptic operator of index $0$. Therefore, $d_2\hat L(\varepsilon, f)$ is an isomorphism if and only if it is injective if and only if it is surjective.
\end{lemma}
\begin{proof}
 For $\phi \in L^p_k\Herm(V, h_0)$, we compute
 \begin{equation}\label{eq:linearizationLhat}
  \begin{split}
   & d_2\hat L(\varepsilon, f)(\phi) = \left.\frac{d}{dt}\right|_{t = 0} \hat L(\varepsilon, f + t\phi) \\
   = & \; \phi\circ K_\alpha^0 + \sqrt{-1}\left(\phi\circ \Lambda^\alpha\bar\partial(f^{-1}\partial_0 f) - f\circ \Lambda^\alpha\bar\partial(f^{-1}\circ \phi \circ \partial_0(f)) + f\circ \Lambda^\alpha\bar\partial(f^{-1}\partial_0(\phi)) \right)\\
   & + \varepsilon \phi\circ \log f + \varepsilon f^{-1}\phi.
  \end{split}
 \end{equation}
 It is clear that $d_2\hat L(\varepsilon, f)$ is a second order differential operator, with the second order term
\[ \sqrt{-1} f\circ \Lambda^\alpha\bar\partial(f^{-1}\partial_0(\phi)) = P^\alpha(\phi) + \mbox{(l.o.t in $\phi$)}. \]
 Thus, $d_2\hat L(\varepsilon, f)$ has the same symbol as $P^\alpha$, for which the statement holds.
\end{proof}

We also have the following counterpart of Proposition 3.2.5 in \cite{LubkeTeleman1995}:
\begin{prop}\label{prop:linearizedestimate} 
 Let $\varepsilon \in (0,1], \sigma \in \RR, f \in L^p_k\Herm^+(V, h_0)$ and $\phi \in \Herm(V, h_0)$. If
 \begin{equation}\label{eq:Lhatandlinearized}
  \hat L(\varepsilon, f) = 0 \text{ and } d_2\hat L(\varepsilon, f)(\phi) + \sigma f\circ \log f = 0,
 \end{equation}
 then, for $\eta := f^{-\half} \phi f^{-\half}$, we have
 $$P^\alpha(|\eta|^2) + 2 \varepsilon|\eta|^2 + \alpha|d_+^f \eta|^2 + (1-\alpha)|d_-^f \eta|^2 \varleq -2\sigma h_0(\log f, \eta).$$
\end{prop}
\begin{proof}
Consider the operators
$P_\pm^f := \sqrt{-1}\Lambda_\pm \bar\partial_\pm^f \partial_{\pm, 0}^f$ 
and 
$P^{\alpha, f} := \alpha P_+^f + (1-\alpha) P_-^f,$
where $\partial_{\pm, 0}^f$ and $\bar\partial_\pm^f$ are given by \eqref{eq:twisted-partial}.
 From \cite{LubkeTeleman1995}, page $68$, we have
 $$\sqrt{-1}d_2(\Lambda_\pm \bar\partial_\pm (f^{-1}\partial_{\pm,0} (f))) = \Ad_f^{-\half}(P_\pm^f(\eta)),$$ 
 which implies
 $$\sqrt{-1}d_2(\Lambda^\alpha \bar\partial (f^{-1}\partial_0 (f))) = \Ad_f^{-\half}(P^{\alpha,f}(\eta)).$$
 Since $\hat L(\varepsilon, f) = f\circ L_\varepsilon(f)$, the first equation in \eqref{eq:Lhatandlinearized} implies that
 $$d_2\hat L(\varepsilon, f)(\phi) = f\circ(\sqrt{-1}d_2(\Lambda^\alpha\bar\partial(f^{-1}\partial_0 (f)))(\phi) + \varepsilon d_2(\log f)(\phi)).$$
 Since $\Ad_f$ acts trivially on $\log f$, the second equation in \eqref{eq:Lhatandlinearized} implies that
 $$P^{\alpha, f}(\eta) + \varepsilon \Phi = - \sigma \log f,$$ 
where $\Phi = \Ad_f^\half (d_2(\log f)(\phi))$.
 Because $\log f$ and $\eta$ are both Hermitian, we get
 $$h_0(P^{\alpha, f}(\eta), \eta) + h_0(\eta, P^{\alpha, f}(\eta)^*) + \varepsilon h_0(\Phi, \eta) + \varepsilon h_0(\eta, \Phi^*) = - 2\sigma h_0(\log f, \eta).$$
 Since $d_\pm^f$ are $h_0$-unitary, as in \cite{LubkeTeleman1995}, page $68$, we get
 $$P_\pm(|\eta|^2) = h_0(P_\pm^f(\eta), \eta) + h_0(\eta, P_\pm^f(\eta)^*) - |d_\pm^f\eta|^2$$ 
 and 
 $$h_0(\Phi, \eta) \vargeq |\eta|^2, h_0(\eta, \Phi^*) \vargeq |\eta|^2.$$
 This implies that
 $$P^\alpha(|\eta|^2) = h_0(P^{\alpha,f}(\eta), \eta) + h_0(\eta, P^{\alpha,f}(\eta)^*) -\alpha|d_+^f\eta|^2 - (1-\alpha)|d_-^f\eta|^2$$
 and the proposition follows.
\end{proof}

We are now ready to prove that $J$ is a nonempty open subset of $(0, 1]$. This is done in the next proposition,
which follows from arguments similar to those used to prove Corollary 3.2.7 in \cite{LubkeTeleman1995}:
\begin{prop}\label{prop:openness} 
 For any $\varepsilon_0 \in (0,1]$ and $f_0 \in \Herm^+(V, h_0)$ such that $\hat L(\varepsilon_0, f_0) = 0$, there exists $\delta > 0$ and a unique differentiable map
 \[\begin{array}{rcl}
 f : I = (0,1] \inter (\varepsilon_0 - \delta, \varepsilon_0 + \delta) & \longrightarrow & \Herm^+(V, h_0) \\ 
 \varepsilon & \longmapsto & f_\varepsilon,
\end{array}\]
 such that $f_{\varepsilon_0} = f_0$ and $L_\varepsilon(f_\varepsilon) = 0$ for all $\varepsilon \in I$. Thus,
 $J$ is a nonempty open subset of $(0, 1]$.
\end{prop}
\begin{proof}
 By the Implicit Function Theorem for Banach spaces and bootstraping for $L_\varepsilon(f) = 0$, 
 we only have to show that $d_2\hat L(\varepsilon, f)$ is an isomorphism at any solution $(\varepsilon, f)$ of $\hat L(\varepsilon, f) =0$ in $(0,1] \times L^p_k\Herm^+(V, h_0)$. 
In fact, proving injectivity of $d_2\hat L(\varepsilon, f)$ suffices by Lemma \ref{lemma:linearization}.

 Suppose that $\phi \in L^p_k\Herm(V,h_0)$ is such that $d_2\hat L(\varepsilon, f)(\phi) = 0$. By Proposition \ref{prop:linearizedestimate}, we get
 $$P^\alpha(|\eta|^2) + 2 \varepsilon|\eta|^2 + \alpha|d_+^f \eta|^2 + (1-\alpha)|d_-^f \eta|^2 \varleq 0,$$
 which implies 
 $$P^\alpha(|\eta|^2) + 2 \varepsilon|\eta|^2 \varleq 0,$$
 where $\eta = f^{-\half} \phi f^{-\half}$. Since $P^\alpha$ is a positive operator, we have
 $$2\varepsilon \int |\eta|^4 d\vol_g \varleq \int P^\alpha(|\eta|^2)|\eta|^2 d\vol_g + 2\varepsilon\int |\eta|^4 d\vol_g\varleq 0.$$
 It follows that $|\eta| = 0$, which implies that $\phi = 0$.
\end{proof}


\subsubsection{Closedness of $J$}\label{subsect:closedness}
We show that $J$ is a closed subset of $(0, 1]$ when $(V, \bar\partial_+, \bar\partial_-)$ is {\em simple} (see definition  \ref{defn:simpleness}).
The closedness of $J$ is then a direct consequence of the following theorem.
\begin{theorem}\label{thm:Jclosed}
 Let $\varepsilon_0 \in (0, 1]$ and suppose there exists $f_\varepsilon \in \Herm^+(V, h_0)$ such that $L_\varepsilon(f_\varepsilon) = 0$
for any $\varepsilon > \varepsilon_0 > 0$ with $\varepsilon \in (0,1]$. 
 If $(V, \bar\partial_+, \bar\partial_-)$ is simple, then  there exists a solution $f_{\varepsilon_0}$ of the equation $L_{\varepsilon_0}(f_{\varepsilon_0}) = 0$.
\end{theorem}


\noindent
To prove Theorem \ref{thm:Jclosed}, we follow the line of argument found in \cite{LubkeTeleman1995}, section \S 3.3. 
The key step of the proof is finding a bound for $\norm{f_\varepsilon}_{L^p_2}$; this is done in Proposition \ref{prop:fLp2boundsbyf1} (5).
We begin by  establishing this and a few other technical results, from which Theorem \ref{thm:Jclosed} will ensue. 

For the rest of this subsection, we work under the hypothesis of Theorem \ref{thm:Jclosed}. 
Let $\varepsilon_0$ and $f_\varepsilon$ be as in the assumption of Theorem \ref{thm:Jclosed}.
We define
\begin{equation}\label{eq:fphiepsilonpair}
 m_\varepsilon := \max_M|\log f_\varepsilon|, \,\,\,\, \phi_\varepsilon := \frac{df_\varepsilon}{d\varepsilon} \,\,\,\, \text{ and } \,\,\,\, \eta_\varepsilon := f_\varepsilon^{-\half} \circ \phi_\varepsilon \circ f^{-\half}.
\end{equation}
Set $$m_K := \max_M |K_\alpha^0|.$$ 
The bound of $\norm{f_\varepsilon}_{L^p_2}$ will then be expressed in terms of $m_K$, $\varepsilon_0$, etc\dots, and will, in particular, be independent of $\varepsilon$
(see Proposition \ref{prop:fLp2boundsbyf1}).

Let us first prove the following counterpart of Lemma 3.3.4 in \cite{LubkeTeleman1995}.
\begin{lemma}\label{lemma:mepsilonbound} 
 Let $\varepsilon \in (0,1]$. Then, for any $f \in \Herm^+(V, h_0)$ such that $L_\varepsilon(f) = 0$, we have
 \begin{enumerate}
  \item $\displaystyle{\half P^\alpha(|\log f|^2) + \varepsilon|\log f|^2 \varleq m_K |\log f|}$,
  \item $\displaystyle{m := \max_M |\log f| \varleq \frac{1}{\varepsilon} m_K}$,
  \item $\displaystyle{m \varleq C \left(\norm{\log f}_{L^2} + m_K\right)}$,
 \end{enumerate}
 where $C$ is a constant depending only on $g$ and $h_0$. In particular, $C$ is independent of $\varepsilon$.
\end{lemma}
\begin{proof}
 $(1)$: Since $0 = L_\varepsilon(f) = K_\alpha^0 + \sqrt{-1} \Lambda^\alpha \bar\partial(f^{-1}\partial_0 f) + \varepsilon \log f$, we have
 $$\varepsilon|\log f|^2 = - h_0(K_\alpha^0, \log f) - h_0(\sqrt{-1}\Lambda^\alpha\bar\partial(f^{-1}\partial_0 f), \log f).$$
 On the other hand, from \cite{LubkeTeleman1995}, page 74,
 $$h_0(\sqrt{-1}\Lambda_\pm\bar\partial_\pm(f^{-1}\partial_{\pm,0} f)) \vargeq \half P_\pm(|\log f|^2) \Longrightarrow h_0(\sqrt{-1}\Lambda^\alpha\bar\partial(f^{-1}\partial_0 f), \log f) \vargeq \half P^\alpha(|\log f|^2).$$
 Cauchy-Schwarz implies
 $$\varepsilon |\log f|^2 \varleq - h_0(K_\alpha^0, \log f) - \half P^\alpha(|\log f|^2) \varleq |K_\alpha^0| |\log f| - \half P^\alpha(|\log f|^2),$$
 which gives $\displaystyle{\half P^\alpha(|\log f|^2) + \varepsilon|\log f|^2 \varleq m_K |\log f|}$.

 $(2)$: Suppose that $|\log f|^2$ is maximal at $x_0 \in M$, that is, $m = |\log f(x_0)|^2$. By the maximum principle, we therefore have
 $$P^\alpha(|\log f|^2) (x_0) \vargeq 0.$$
 Then, $(1)$ implies that
 $$\varepsilon |\log f(x_0)|^2 \varleq m_K |\log f(x_0)| \Longrightarrow m \varleq \frac{1}{\varepsilon} m_K.$$

 $(3)$: By $(1)$, we have
 $$P^\alpha(|\log f|^2) \varleq 2 m_K |\log f| \varleq m_K^2 + |\log f|^2.$$
 Since $P^\alpha$ is elliptic, it follows from Lemma 3.3.2 in \cite{LubkeTeleman1995} that
 $$m^2 = \max_M |\log f|^2 \varleq C(\norm{|\log f|^2}_{L^1} + m_K^2) = C(\norm{\log f}_{L^2}^2 + m_K^2) \varleq C(\norm{\log f}_{L^2} + m_K)^2.$$
We thus get $m \varleq C(\norm{\log f}_{L^2} + m_K)$.
\end{proof}

The following lemma is similar to Lemma 3.3.1 in \cite{LubkeTeleman1995}.
\begin{lemma}\label{lemma:boundetasquare}
 Let $(V, \bar\partial_+, \bar\partial_-)$ be simple. For any $\varepsilon \in (\varepsilon_0, 1]$, there exists a constant $C := C(m_\varepsilon)$ depending only on $m_\varepsilon$ such that
 $$\alpha\norm{d_+^{f_\varepsilon} \eta_\varepsilon}_{L^2}^2 + (1-\alpha)\norm{d_-^{f_\varepsilon} \eta_\varepsilon}_{L^2}^2 \vargeq C\norm{\eta_\varepsilon}_{L^2}^2.$$
\end{lemma}
\begin{proof}
 As in the proof of Lemma 3.3.1 in \cite{LubkeTeleman1995}, we define $\psi_\varepsilon := \Ad_{f_\varepsilon}^{-\half} \eta_\varepsilon$ and obtain
 $$\norm{d_\pm^{f_\varepsilon} \eta_\varepsilon}_{L^2}^2 \vargeq C\norm{\bar\partial_\pm \psi_\varepsilon}_{L^2}^2 = C \<\bar\partial_\pm^* \bar\partial_\pm \psi_\varepsilon, \psi_\varepsilon\>_{L^2} = C\<\Laplacian_{\bar\partial_\pm} \psi_\varepsilon, \psi_\varepsilon\>_{L^2,}$$
 where $\Laplacian_{\bar\partial_\pm} = \bar\partial_\pm^* \bar\partial_\pm$ and $C$ depends only on $m_\varepsilon$. Let $\Laplacian_{\bar\partial, \alpha} := \alpha\Laplacian_{\bar\partial_+} + (1-\alpha)\Laplacian_{\bar\partial_-}$. We then have
 $$\alpha\norm{d_+^{f_\varepsilon} \eta_\varepsilon}_{L^2}^2 + (1-\alpha)\norm{d_-^{f_\varepsilon} \eta_\varepsilon}_{L^2}^2 \vargeq C\<\Laplacian_{\bar\partial,\alpha} \psi_\varepsilon, \psi_\varepsilon\>_{L^2}.$$

 For any $\varphi \in C^\infty(\End(V))$, we compute
 $$\<\Laplacian_{\bar\partial,\alpha} \varphi, \varphi\>_{L^2} = \alpha \norm{\bar\partial_+\varphi}_{L^2}^2 +(1-\alpha) \norm{\bar\partial_-\varphi}_{L^2}^2.$$
Recall from definition \ref{defn:simpleness} that $(V, \bar\partial_+, \bar\partial_-)$ is simple if, for any $\Phi \in C^\infty(\End(V))$
with $\bar \partial_\pm \Phi = 0$, we have $\Phi = c \id_V$ for some $c \in \CC$.
The simpleness of $(V, \bar\partial_+, \bar\partial_-)$ thus implies that $\ker (\Laplacian_{\bar\partial,\alpha}) = c \id_V$. 
 
 Let $c_1$ be the first non-zero eigenvalue of $\Laplacian_{\bar\partial,\alpha}$. Then $c_1 >0$ since $\Laplacian_{\bar\partial,\alpha}$ is non-negative. Moreover, $c_1$ is independent of $\varepsilon$ and
 $$\<\Laplacian_{\bar\partial,\alpha} \varphi, \varphi\>_{L^2} \vargeq c_1\norm{\varphi}_{L^2}^2$$ 
 for any $\varphi$ such that $$\int_M \tr(\varphi) d\vol_g = 0.$$
 Since $|\psi_\varepsilon| = \left|\Ad_{f_\varepsilon}^{-\half} \eta_\varepsilon\right| \vargeq C(m_\varepsilon) |\eta_\varepsilon|$ and
 $$\displaystyle{\tr(\psi_\varepsilon) = \tr(\eta_\varepsilon) = \frac{d}{d\varepsilon} \tr(\log f_\varepsilon) = \frac{d}{d\varepsilon} \log \det f_\varepsilon = 0},$$
 we obtain
 $$\alpha\norm{d_+^{f_\varepsilon} \eta_\varepsilon}_{L^2}^2 + (1-\alpha)\norm{d_-^{f_\varepsilon} \eta_\varepsilon}_{L^2}^2 \vargeq C\<\Laplacian_{\bar\partial,\alpha} \psi_\varepsilon, \psi_\varepsilon\>_{L^2} \vargeq C c_1 \norm{\psi_\varepsilon}_{L^2}^2 \vargeq C'\norm{\eta_\varepsilon}_{L^2}^2$$
 and $C'$ only depends on $m_\varepsilon$.
\end{proof}

The next proposition combines Propositions 3.3.3 and 3.3.5 in \cite{LubkeTeleman1995}.
\begin{prop}\label{prop:fLp2boundsbyf1}
 Suppose $(V, \bar\partial_+, \bar\partial_-)$ is simple and there exists a positive constant $m$ such that $m_\varepsilon \varleq m$ for all $\varepsilon \in (\varepsilon_0, 1]$. 
 Then, there exists a constant $C: = C(m)$ such that 
 \begin{enumerate}
  \item $\displaystyle{\max_M |\phi_\varepsilon| \varleq C(m)}$,
  \item $\displaystyle{\norm{\phi_\varepsilon}_{L^p} \varleq C(m)\left(\norm{\phi_\varepsilon}_{L^p} + \norm{P^\alpha(\phi_\varepsilon)_{L^p}}\right)}$,
  \item $\displaystyle{\norm{P^\alpha(\phi_\varepsilon)}_{L^p} \varleq C(m) \left(1 + \norm{\phi_\varepsilon}_{L^{2p}_1} \norm{f_\varepsilon}_{L^{2p}_1} + \norm{f_\varepsilon}_{L^{2p}_1}^2\right)}$,
  \item $\displaystyle{\norm{\phi_\varepsilon}_{L^p_2} \varleq C(m) \left(1 + \norm{f_\varepsilon}_{L^p_2}\right)}$,
  \item $\displaystyle{\norm{f_\varepsilon}_{L^p_2} \varleq e^{C(m)(1-\varepsilon)} \left(1 + \norm{f_1}_{L^p_2}\right)}$,
 \end{enumerate}
for any $p >1$ and $\varepsilon \in (\varepsilon_0, 1]$.
\end{prop}
\begin{proof}
 $(1)$: By Proposition \ref{prop:openness}, we may differentiate $\hat L(\varepsilon, f_\varepsilon) = 0$ with respect to $\varepsilon$ and get
 $$d_2\hat L(\varepsilon, f_\varepsilon)(\phi_\varepsilon) + f_\varepsilon \circ \log f_\varepsilon = 0.$$
 By Proposition \ref{prop:linearizedestimate}, with $\sigma = 1$, we have
 \begin{equation}\label{eq:Peta}
  P^\alpha(|\eta_\varepsilon|^2) + 2 \varepsilon|\eta_\varepsilon|^2 + \alpha|d_+^f \eta_\varepsilon|^2 + (1-\alpha)|d_-^f \eta_\varepsilon|^2 \varleq -2 h_0(\log f_\varepsilon, \eta_\varepsilon) \varleq 2 |\log f_\varepsilon||\eta_\varepsilon|.
 \end{equation}

 Since $P^\alpha$ is a positive operator, integrating \eqref{eq:Peta} gives
 $$\alpha \norm{d_+^f \eta_\varepsilon}_{L^2}^2 + (1-\alpha) \norm{d_-^f \eta_\varepsilon}_{L^2}^2 \varleq - 2 \<\log f_\varepsilon, \eta_\varepsilon\>_{L^2} \varleq 2 \norm{\log f}_{L^2} \norm{\eta_\varepsilon}_{L^2} \varleq C_1(m) \norm{\eta_\varepsilon}_{L^2}.$$
Lemma \ref{lemma:boundetasquare} then implies that
 $$C \norm{\eta_\varepsilon}_{L^2}^2 \varleq C_1(m) \norm{\eta_\varepsilon}_{L^2} \Longrightarrow \norm{\eta_\varepsilon}_{L^2} \varleq C'$$
 for some constant $C'$ that depends only on $m$.

 Again, by equation \eqref{eq:Peta}, we have
 $$P^\alpha(|\eta_\varepsilon|^2)  \varleq 2 |\log f_\varepsilon||\eta_\varepsilon|\varleq |\eta_\varepsilon|^2 + m^2.$$
 As in the proof of Lemma \ref{lemma:mepsilonbound} $(3)$, we get
 $$\max_M |\eta_\varepsilon| \varleq C_2 \left(\norm{\eta_\varepsilon}_{L^2} + m\right) \Longrightarrow \max_M |\phi_\varepsilon| \varleq C(m).$$

 $(2)$: This follows from the elliptic estimate for the elliptic operator $P^\alpha$ (see Theorem 9.11 of \cite{GilbargTrudinger1998}).

 $(3)$: The proof on pages 76 and 77 of \cite{LubkeTeleman1995} works here. Differentiating $L_\varepsilon(f_\varepsilon) = 0$ with respect to $\varepsilon$ gives the same identity for $P^\alpha$ (instead of $P$). The rest of the proof is identical.

 $(4)$: This follows from $(1) - (3)$.

 $(5)$: This follows from $(4)$. Indeed, for $t \in [\varepsilon, 1]$, let $x(t) = \norm{f_t}_{L^p_2} $. Then,
 $$\frac{d}{d t}x(t) \vargeq - \norm{\frac{d}{dt} f_t}_{L^p_2} \vargeq - C(m)(1 + x(t)) \Longrightarrow \frac{x'(t)}{1+x(t)} \vargeq -C(m).$$
 Integrating over $[\varepsilon, 1]$, we get
 $$\frac{1+x(1)}{1+x(\varepsilon)} \vargeq e^{-C(m)(1-\varepsilon)} \Longrightarrow \norm{f_\varepsilon}_{L^p_2} \varleq 1 + \norm{f_\varepsilon}_{L^p_2} \varleq e^{C(m)(1 - \varepsilon)}\left(1 + \norm{f_1}_{L^p_2}\right).$$
\end{proof}

\noindent
{\it Proof of Theorem \ref{thm:Jclosed}.}
The item $(5)$ above gives a bound of $\norm{f}_{L^p_2}$. From this point on, the proof of Theorem \ref{thm:Jclosed} is identical to the proof of Proposition 3.3.6 (i) in \cite{LubkeTeleman1995}.
\qed

We end this subsection by stating the following corollary, which is the analogue of Proposition 3.3.6 (ii) in \cite{LubkeTeleman1995}.
\begin{corollary}\label{coro:L2boundgiveslimit}
 Suppose that there exists a constant $C$ such that $\norm{f_\varepsilon}_{L^2} \varleq C$ for all $\varepsilon \in (0, 1]$. 
 Then, there exists a solution $f_0$ of the equation $L_0(f_0) = 0$. In other words, the metric $h$ defined by $h(s, t) := h_0(f_0(s), t)$ is  $\alpha$-Hermitian-Einstein. 
\end{corollary}
\begin{proof}
The proof is similar to the proof of Proposition 3.3.6 (ii) in \cite{LubkeTeleman1995}
\end{proof}


\subsubsection{Destabilizing subsheaves}\label{subsect:coherentsheavesIpm}
We consider the limit when $\varepsilon \to 0$ and show that $\alpha$-stability (see definition \ref{defn:alphastable}) gives a bound on $\norm{f_\varepsilon}_{L^2}$, as required by Corollary \ref{coro:L2boundgiveslimit}. More precisely, we prove
\begin{prop}\label{prop:destabilizing}
 If $\limsup\limits_{\varepsilon\rightarrow 0} \norm{\log f_\varepsilon}_{L^2} = \infty$, then $(V, \bar\partial_+, \bar\partial_-)$ is not $\alpha$-stable.
\end{prop}

\noindent
This corresponds to Proposition 3.4.1 in \cite{LubkeTeleman1995}. 

To prove Proposition \ref{prop:destabilizing}, we again follow \cite{LubkeTeleman1995}, this time section \S 3.4, 
to show the existence of a \emph{destabilizing} coherent subsheaf of $(V, \bar\partial_+, \bar\partial_-)$ under the assumptions of Proposition \ref{prop:destabilizing}.
We first recall, from definition \ref{defn:coherentsheaves}, the key notion of \emph{coherent subsheaf} in this context.
 \begin{defn*}
 Let $\FFC_\pm$ be coherent subsheaves of $(V, \bar\partial_\pm)$, respectively. The pair $\FFC := (\FFC_+, \FFC_-)$ is said to be a \emph{coherent subsheave} of $(V, \bar\partial_+, \bar\partial_-)$  if there exists analytic subsets $S_+$  and $S_-$ of $(M, I_+)$ and $(M, I_-)$, respectively, such that
  \begin{enumerate}
   \item $S := S_-\union S_+$ have codimension at least $2$;
   \item $\FFC_\pm|_{M \setminus S_\pm}$ are locally free and $\FFC_-|_{M\setminus S} = \FFC_+|_{M\setminus S}$.
  \end{enumerate}
The \emph{rank} of $\FFC$ is the rank of $\FFC_\pm|_{M \setminus S_\pm}$ and the \emph{$\alpha$-slope} of $\FFC$ is
  $$\mu_\alpha(\FFC) := \alpha \frac{\deg_+(\FFC_+)}{\rk(\FFC_+)} + (1-\alpha) \frac{\deg_-(\FFC_-)}{\rk(\FFC_-)}.$$
\end{defn*}

We also consider the following objects.

\begin{defn}\label{defn:weaklyholomorphic}
 An endomorphism $\pi \in L^2_1(\End(V))$ is a \emph{weakly holomorphic subbundle} of $(V, \bar\partial_\pm)$ if
 $$\pi^2 = \pi^* = \pi \text{ and } \left(\id_V - \pi\right) \circ \bar\partial_\pm(\pi) = 0 \in L^1(\End(V)).$$
\end{defn}

\begin{remark*}
Notice that, since $M$ is compact, we have $L^2(\End(V)) \into L^1(\End(V))$ and
$$L^2(\End(V)) \times L^2(\End(V)) \to L^1(\End(V)) : (f, g) \mapsto f\circ g.$$
\end{remark*}

The following result of Uhlenbeck-Yau, which is stated as Theorem 3.4.3 in \cite{LubkeTeleman1995}, then tells us that weakly holomorphic subbundles correspond to coherent subsheaves.
\begin{theorem}[Uhlenbeck-Yau \cite{UhlenbeckYau86,UhlenbeckYau89}]
A weakly holomorphic subbundle $\pi$ of $V$ represents a coherent subsheaf $\FFC$ of $V$. More precisely, there exists a coherent subsheaf $\FFC$ of $V$ and an analytic subset $S \subset M$ of codimension at least $2$ such that $\FFC|_{M \setminus S}$ is a holomorphic subbundle of $V|_{M\setminus S}$ and the orthogonal projection of $E|_{M\setminus S}$ onto $\FFC|_{M \setminus S}$ is given by $\pi$.
\end{theorem}
\noindent
By the Uhlenbeck-Yau Theorem, we then only need to construct a simultaneous weakly holomorphic subbundle $\pi$ of both $(V,\bar\partial_\pm)$ that destabilizes $(V, \bar\partial_+, \bar\partial_-)$. 
This will occupy the rest of this section.
We start with the following lemma, which corresponds to Lemma 3.4.4 in \cite{LubkeTeleman1995}.
\begin{lemma}\label{lemma:termsestimate}
 Let $f \in \Herm^+(V, h_0)$ and $0 < \sigma \varleq 1$. Then,
 \begin{enumerate}
  \item $\displaystyle{\sqrt{-1}\Lambda_\pm h_0(f^{-1}\partial_{\pm,0} f, \partial_{\pm,0}(f^\sigma)) \vargeq \left|f^{-\frac{\sigma}{2}}\partial_{\pm,0}(f^\sigma)\right|^2}$.
  \item Suppose further that $L_\varepsilon(f) = 0$ for some $\varepsilon >0$. Then,
  $$\frac{1}{\sigma}P^\alpha(\tr f^\sigma) + \varepsilon h_0(\log f, f^\sigma) + \alpha \left|f^{-\frac{\sigma}{2}}\partial_{+,0}(f^\sigma)\right|^2 + (1-\alpha) \left|f^{-\frac{\sigma}{2}}\partial_{-,0}(f^\sigma)\right|^2 \varleq - h_0(K_\alpha^0, f^\sigma).$$
 \end{enumerate}
\end{lemma}
\begin{proof}
 The first part follows from Lemma 3.4.4 (i) of \cite{LubkeTeleman1995}. The second part is similar, but we give the details below. Since $L_\varepsilon(f) = K_\alpha^0 + \sqrt{-1} \Lambda^\alpha \bar\partial(f^{-1}\partial_0 f) + \varepsilon \log f = 0$, we have
  $$h_0(K_\alpha^0, f^\sigma) + \sqrt{-1} h_0(\Lambda^\alpha \bar\partial(f^{-1}\partial_0 f), f^\sigma) + \varepsilon h_0(\log f, f^\sigma) = 0.$$
 The second term above can be rewritten as
 $$h_0(\Lambda^\alpha \bar\partial(f^{-1}\partial_0 f), f^\sigma) = \Lambda^\alpha \bar\partial h_0(f^{-1}\partial_0 f, f^\sigma) + \Lambda^\alpha h_0(f^{-1}\partial_0 f, \partial_0 f^\sigma),$$
 where the terms on the right denote the $\alpha$-linear combinations of the respective terms with ${}_\pm$-subscripts. More precisely, we have
 \begin{equation*}
  \begin{split}
   & \sqrt{-1}\Lambda^\alpha h_0(f^{-1}\partial f, \partial f^\sigma) \\
   & := \sqrt{-1}\alpha\Lambda_+ h_0(f^{-1}\partial_{+,0} f, \partial_{+,0} f^\sigma) + \sqrt{-1}(1-\alpha)\Lambda_- h_0(f^{-1}\partial_{-,0} f, \partial_{-,0} f^\sigma) \\
   & \vargeq \alpha \left|f^{-\frac{\sigma}{2}}\partial_{+,0}(f^\sigma)\right|^2 + (1-\alpha)\left|f^{-\frac{\sigma}{2}}\partial_{-,0}(f^\sigma)\right|^2
  \end{split}
 \end{equation*}
 by part $(1)$, and
 \begin{equation*}
  \begin{split}
   & \sqrt{-1}\Lambda^\alpha h_0(f^{-1}\partial_0 f, \partial_0 f^\sigma) \\
   & := \alpha\sqrt{-1}\Lambda_+ h_0(f^{-1}\partial_{+,0} f, \partial_{+,0} f^\sigma) + (1-\alpha)\sqrt{-1}\Lambda_- h_0(f^{-1}\partial_{-,0} f, \partial_{-,0} f^\sigma) \\
   & = \frac{1}{\sigma} \left(\alpha P_+(\tr f^\sigma) + (1-\alpha) P_-(\tr f^\sigma) \right) = \frac{1}{\sigma} P^\alpha(\tr f^\sigma).
  \end{split}
 \end{equation*}
 The inequality in $(2)$ then follows.
\end{proof}

For the rest of this subsection, we assume that the assumptions of Proposition \ref{prop:destabilizing} hold, that is,
$$\limsup\limits_{\varepsilon\rightarrow 0} \norm{\log f_\varepsilon}_{L^2} = \infty.$$
For any $\varepsilon > 0$ and $x \in M$, let $\lambda(\varepsilon, x)$ be the largest eigenvalue of $\log f_\varepsilon$, 
$M_\varepsilon:= \max\{\lambda(\varepsilon, x): x\in M\}$ 
and 
$$\rho(\varepsilon) := e^{-M_\varepsilon}.$$
The next proposition corresponds to parts Proposition $3.4.6$ (i) and (ii) in \cite{LubkeTeleman1995}.
\begin{prop}\label{prop:weaklyconvergence}
 There exists a subsequence $\varepsilon_i$ of positive real numbers such that
  $$\lim\limits_{i\to \infty}\varepsilon_i = 0 \text{ and } \lim\limits_{i \to \infty} \norm{\log f_{\varepsilon_i}}_{L^2} = \infty.$$
 Let $f_i := \rho(\varepsilon_i)f_{\varepsilon_i}$. Then, in $L^2_1(\Herm(V, h_0))$, we have:
 \begin{enumerate}
  \item There exists $0 \neq f_\infty$ such that $f_i \weakto f_\infty$ weakly in $L^2_1$.
  \item There exist $0 \neq \tilde f_\infty$ and a sequence $\sigma_j \to 0$ in $(0,1)$ such that $\displaystyle{f_\infty^{\sigma_j} \weakto \tilde f_\infty}$ weakly in $L^2_1$.
 \end{enumerate}
\end{prop}
\begin{proof}
As in the proof of Proposition $3.4.6$ (i) in \cite{LubkeTeleman1995}, the key is showing that $\rho(\varepsilon) f_\varepsilon$ is bounded in $L^2_1$, 
which leads to statement $(1)$. Statement $(2)$ can then be proven using the same arguments, applied to $\left(\rho(\varepsilon)f_\varepsilon\right)^\sigma$. We do it in two steps.

 {\bf Step I}:
 We show that there exists $C > 0$ (independent of $\varepsilon$) such that
 $$C^{-1} \varleq \norm{\rho(\varepsilon)f_\varepsilon}_{L^2} \varleq C.$$
To do this, we apply Lemma \ref{lemma:termsestimate} $(2)$, with $\sigma = 1$, and get
 $$P^\alpha(\tr f_\varepsilon) + \varepsilon h_0(\log f_\varepsilon, f_\varepsilon) + \alpha \left|f_\varepsilon^{-\frac{1}{2}}\partial_{+,0} f_\varepsilon\right|^2 + (1-\alpha) \left|f_\varepsilon^{-\frac{1}{2}}\partial_{-,0}f_\varepsilon\right|^2 \varleq - h_0(K_\alpha^0, f_\varepsilon).$$
 It follows that
 $$P^\alpha(\tr f_\varepsilon) \varleq -h_0(\varepsilon \log f_\varepsilon + K_\alpha^0, f_\varepsilon) \varleq C \max_M(|\varepsilon \log f_\varepsilon| + |K_\alpha^0|) |f_\varepsilon| \varleq C |f_\varepsilon|,$$
 where the last inequality uses Lemma \ref{lemma:mepsilonbound} $(2)$.

 Since $f_\varepsilon \in \Herm^+(V, h_0)$, we have
 $$C_1^{-1} \tr f_\varepsilon \varleq |f| \varleq C_1 \tr f_\varepsilon,$$
 which implies that $P^\alpha(\tr f_\varepsilon) \varleq C \tr f_\varepsilon$.
 Moser's iteration implies that
 $$\norm{\tr f_\varepsilon}_{L^\infty} \varleq C \norm{\tr f_\varepsilon}_{L^2} \Longrightarrow \norm{\rho(\varepsilon)f_\varepsilon}_{L^\infty} \varleq C \norm{\rho(\varepsilon)f_\varepsilon}_{L^2}$$
 for some $C > 0$, independent of $\varepsilon$.

 By the definition of $\rho(\varepsilon)$, we get
 $$\norm{\rho(\varepsilon) f_\varepsilon}_{L^\infty} \vargeq 1 \Longrightarrow C^{-1} \varleq \norm{\rho(\varepsilon) f_\varepsilon}_{L^2}$$
  and, on the other hand,
 $$\norm{\rho(\varepsilon) f_\varepsilon}_{L^2} \varleq \Vol_g(M)^{\frac{1}{2}} \norm{\rho(\varepsilon)f_\varepsilon}_{L^\infty} \varleq \Vol_g(M)^{\frac{1}{2}} \norm{\id_V}_{L^\infty} \varleq C.$$

 {\bf Step II}: Let $\Cnabla{+}$ be the Chern connection for $\bar\partial_+$ defined with respect to $h_0$. We show that
  $$\norm{\Cnabla{+}(\rho(\varepsilon)f_\varepsilon)}_{L^2} \varleq C,$$
  which implies that $\rho(\varepsilon)f_\varepsilon$ is bounded in $L^2_1$.
  Indeed,
 \begin{equation*}
  \begin{split}
   & \norm{\Cnabla{+}(\rho(\varepsilon)f_\varepsilon)}_{L^2} = \int_M \left|\Cnabla{+}(\rho(\varepsilon)f_\varepsilon)\right|^2 d\vol_g = 2 \int_M \left|\partial_{+,0}(\rho(\varepsilon)f_\varepsilon)\right|^2 d\vol_g \\
   \varleq \, & 2 \int_M \left|\left(\rho(\varepsilon) f_\varepsilon\right)^{-\frac{1}{2}}\partial_{+,0}(\rho(\varepsilon)f_\varepsilon)\right|^2 d\vol_g \hspace{0.2in} (\text{ since } \rho(\varepsilon) f_\varepsilon \varleq \id_V) \\
   \varleq \, & 2 \rho(\varepsilon) \alpha^{-1} \int_M \left(\alpha \left|f_\varepsilon^{-\frac{1}{2}}\partial_{+,0} f_\varepsilon\right|^2 + (1-\alpha) \left|f_\varepsilon^{-\frac{1}{2}}\partial_{-,0}f_\varepsilon\right|^2 \right) d\vol_g \\
   \varleq \, & 2 \alpha^{-1} \int_M \left(-h_0(K_\alpha^0, \rho(\varepsilon)f_\varepsilon) - \rho(\varepsilon)P^\alpha(\tr f_\varepsilon) - h_0(\varepsilon \log f_\varepsilon, \rho(\varepsilon) f_\varepsilon)\right) d\vol_g \\
   \varleq \, & C.
  \end{split}
 \end{equation*}
 In the last inequality, the first and third terms follow from Step I and Lemma \ref{lemma:mepsilonbound}. Moreover, the second term vanishes because $g$ is Gauduchon with respect to both $I_\pm$:
 $$\int_M P^\alpha(\tr f_\varepsilon) d\vol_g = \alpha \int_M P_+(\tr f_\varepsilon) \omega_+^n + (1-\alpha) \int_M P_-(\tr f_\varepsilon) \omega_-^n = 0.$$
 The non-vanishing of $f_\infty$ in $(1)$ follows from the lower bound in Step I.
\end{proof}

\noindent
{\em Proof of Proposition \ref{prop:destabilizing}.}
We follow the proof of Proposition 3.4.1 in \cite{LubkeTeleman1995}, pages $86$ -- $90$.
From Proposition \ref{prop:weaklyconvergence}, the following defines a projection over a subset $W \subset M$ of full measure:
$$\pi = \id_V - \lim\limits_{\sigma_j \to 0}\left(\lim\limits_{i \to \infty} f_i\right)^{\sigma_j} \in L^1(\Herm(V, h_0))$$
The convergence is pointwise in $W$ because weakly convergence in $L^2_1$ implies almost everywhere pointwise convergence. We then need to show that
$$(\id_V - \pi) \circ \bar\partial_\pm(\pi) = 0$$
in $L^1$.
As in \cite{LubkeTeleman1995}, we have
$$\left|(\id_V - \pi) \circ \bar\partial_\pm(\pi)\right| = \left| \bar\partial_\pm(\id_V - \pi) \circ \pi\right| = \left| \left(\bar\partial_\pm(\id_V - \pi) \circ \pi\right)^*\right| = \left| \pi \circ \bar\partial_{\pm,0}(\id_V - \pi)\right|$$
and it suffices to show that
\begin{equation}\label{eq:projectionL2vanish}
 \norm{\pi\circ \bar\partial_{\pm,0}(\id_V - \pi)}_{L^2} = 0.
\end{equation}
Since $f_i \varleq \id_V$, that is, all eigenvalues of $f_i$ are in $(0,1]$, the following inequality holds
$$0 \varleq \frac{a + \frac{b}{2}}{a}\left(\id_V - f_i^a\right) \varleq f_i^{-\frac{b}{2}}$$
as in \cite{LubkeTeleman1995}.
Using Lemma \ref{lemma:termsestimate} $(2)$, we have
\begin{equation*}
 \begin{split}
  & \int_M\left(\alpha\left|(\id_V - f_i^a) \circ \partial_+(f_i^b)\right|^2 + (1-\alpha) \left|(\id_V - f_i^a) \circ \partial_-(f_i^b)\right|^2 \right) d\vol_g \\
  \varleq \, & \frac{a}{a+\frac{b}{2}} \int_M\left( \alpha\left|f_i^{-\frac{b}{2}} \circ \partial_+(f_i^b)\right|^2 + (1-\alpha)\left|f_i^{-\frac{b}{2}}\circ \partial_-(f_i^b)\right|^2\right)d\vol_g  \\
  \varleq \, & \frac{a}{a+\frac{b}{2}} \int_M \left(-h_0(K_\alpha^0, f_i) - P^\alpha(\tr f_i) - h_0(\varepsilon_i \log f_{\varepsilon_i}, f_i)\right) d\vol_g \\
   \varleq \, & \frac{a}{a+\frac{b}{2}} C  \hspace{0.5in} (\text{as in Step II of Proposition \ref{prop:weaklyconvergence}}).
 \end{split}
\end{equation*}
Taking limits in the order $\lim\limits_{b \to 0} \lim\limits_{a \to 0} \lim\limits_{i \to \infty}$ on both sides of the above inequality, we get
$$\int_M\left(\alpha\left|(\id_V - \pi) \circ \partial_+(\pi)\right|^2 + (1-\alpha) \left|(\id_V - \pi) \circ \partial_-(\pi)\right|^2 \right) d\vol_g = 0,$$
which implies equation \eqref{eq:projectionL2vanish}. Hence, $\pi$ represents a coherent subsheaf $\mc F = (\mc F_+, \mc F_-)$ of $(V, \bar\partial_+, \bar\partial_-)$.
In fact, we have
$$0 < \rk(\mc F) := \rk(\mc F_\pm) < \rk(V),$$
so that $\mc F$ is a proper coherent subsheaf of $(V, \bar\partial_+, \bar\partial_-)$; this is proven the same way as Corollary 3.4.7 in \cite{LubkeTeleman1995}.
To show that $\mc F$ destabilizes $(V, \bar\partial_+, \bar\partial_-)$, that is,
$$\mu_\alpha(\mc F) \vargeq \mu_\alpha(V),$$
we carry out the proof of Proposition 3.4.8 in \cite{LubkeTeleman1995} verbatim, using the corresponding fact that
$$\mu^\alpha(\mc F) = \mu^\alpha(V) + \frac{(n-1)!}{2\pi \rk \mc F} \int_{W} \left(\tr(K_\alpha^0\circ \pi) - \alpha\left|\partial_+(\pi)\right|^2 - (1-\alpha)\left|\partial_-(\pi)\right|^2\right) d\vol_g.$$

\qed

\end{document}